\documentclass{article}

\usepackage{amsfonts,amsmath,amssymb,amsthm}
\usepackage[algo2e,ruled,vlined,linesnumbered,inoutnumbered]{algorithm2e}
\usepackage{graphicx}
\usepackage{epstopdf}
\usepackage[caption=false]{subfig}
\usepackage{url}
\usepackage{cases}
\usepackage{booktabs}
\usepackage{appendix}
\usepackage[top=2.5cm, bottom=2.5cm, left=3cm, right=3cm]{geometry}
\usepackage[breaklinks,colorlinks,linkcolor=black,citecolor=black,urlcolor=blue]{hyperref}

\newtheorem{theorem}{Theorem}[section]%
\newtheorem{proposition}[theorem]{Proposition}%
\newtheorem{lemma}[theorem]{Lemma}%
\newtheorem{corollary}[theorem]{Corollary}%
\newtheorem{condition}{Condition}%
\newtheorem{definition}[theorem]{Definition}
\newtheorem{remark}{Remark}

\newcommand{\rank}{\mbox{\rm rank}}

\newcommand{\bR}{\mathbb{R}}
\newcommand{\N}{\mathbb{N}}

\newcommand{\Rnp}{\mathbb{R}^{n\times p}} 
 
\newcommand{\Rpp}{\mathbb{R}^{p\times p}} 
\newcommand{\Rnn}{\mathbb{R}^{n\times n}}  
\newcommand{\Rnm}{\mathbb{R}^{n\times m}}

\newcommand{\cL}{\mathcal{L}}

\newcommand{\stiefel}{\mathcal{S}_{n,p}}

\newcommand{\tr}{\mathrm{tr}}
\newcommand{\zz}{^{\top}}
\newcommand{\inv}{^{-1}}
\newcommand{\st}{\mathrm{s.\,t.}\,\,} 

\newcommand{\ff}{_{\mathrm{F}}}
\newcommand{\fs}{^2_{\mathrm{F}}}

\newcommand{\dkh}[1]{\left(#1\right)}
\newcommand{\hkh}[1]{\left\{#1\right\}}
\newcommand{\fkh}[1]{\left[#1\right]}
\newcommand{\jkh}[1]{\left\langle#1\right\rangle}
\newcommand{\norm}[1]{\left\|#1\right\|}
\newcommand{\abs}[1]{\left\lvert #1 \right\rvert}

\newcommand{\proj}{\mathrm{Proj}}
\newcommand{\projts}[2]{\mathrm{Proj}_{\mathcal{T}_{#1} \mathcal{S}_{n,p}} \dkh{#2}}
\newcommand{\prox}[2]{\mathrm{Prox}_{#1} \dkh{#2}}

\newcommand{\tanst}[1]{\mathcal{T}_{#1} \mathcal{S}_{n,p}}

\newcommand{\dsp}[2]{{\mathbf{D_p}}\left(#1,#2\right)}
\newcommand{\distp}[2]{{\mathbf{d_p}}\left(#1,#2\right)}
\newcommand{\Pv}[1]{{\mathbf{P}}^{\perp}_{#1}}
\newcommand{\dists}[2]{{\mathbf{d}^2_{\mathbf{p}}}\left(#1,#2\right)}

\newcommand{\iid}{i=1,\dotsc,d}

\newcommand{\iin}{i=1,\dotsc,n}
\newcommand{\sumiid}{\sum\limits_{i=1}^d}

\newcommand{\sumjjd}{\sum\limits_{j=1}^d}

\DeclareMathOperator*{\argmin}{arg\,min}

\begin{document}

\title{A Communication-Efficient And Privacy-Aware 
	Distributed Algorithm For Sparse PCA}

\author{Lei Wang\thanks{State Key Laboratory of Scientific and Engineering 
		Computing, Academy of Mathematics and Systems Science, Chinese  Academy of 
		Sciences, and University of Chinese Academy of Sciences, China 
		(wlkings@lsec.cc.ac.cn). Research is supported by the National Natural Science 
		Foundation of China (No. 11971466 and 11991020).}
	\and Xin Liu\thanks{State Key Laboratory of Scientific and Engineering 
		Computing, Academy of Mathematics and Systems Science, Chinese Academy of 
		Sciences, 
		and University of Chinese Academy of Sciences, China (liuxin@lsec.cc.ac.cn). 
		Research is supported in part by the National Natural Science Foundation of 
		China (No. 12125108, 11991021, and 12288201), Key Research Program of Frontier 
		Sciences, Chinese Academy of Sciences (No. ZDBS-LY-7022).}
	\and 
	Yin Zhang\thanks{School of Data Science, The Chinese University of Hong Kong, Shenzhen, China (yinzhang@cuhk.edu.cn). 
	Research is supported in part by 
	the Shenzhen Science and Technology Program 
	(No. GXWD20201231105722002-20200901175001001).
	}
}

\date{}

\maketitle

\begin{abstract}
	Sparse principal component analysis (PCA) improves interpretability of the classic PCA by 
	introducing sparsity into the dimension-reduction process.  Optimization models for 
	sparse PCA, however, are generally non-convex, non-smooth and more difficult to solve, 
	especially on large-scale datasets requiring distributed computation over a wide network.
	In this paper, we develop a distributed and centralized algorithm called DSSAL1 
	for sparse PCA that aims to achieve low communication overheads by adapting 
	a newly proposed subspace-splitting strategy to accelerate convergence.  
	Theoretically, convergence to stationary points is established for DSSAL1.
	Extensive numerical results show that DSSAL1 requires far fewer rounds of communication than 
	state-of-the-art peer methods.
	In addition, we make the case that since messages exchanged in DSSAL1 are well-masked, 
	the possibility of private-data leakage in DSSAL1 is much lower than in some other
	distributed algorithms.
\end{abstract}

\section{Introduction}

\label{sec:introduction}

	The principal component analysis (PCA) is a fundamental and ubiquitous tool 
	in statistics and data analytics. 
	In particular, it frequently serves as a critical preprocessing step 
	in numerous data science or machine learning tasks.
	In general, solutions produced by the classical PCA are dense combinations of all features. 
	However, in practice, 
	sparse combinations not only enhance the interpretability of the principal components 
	but also reduce storage \cite{Sjostrand2007sparse,Chen2013biclustering}, 
	which motivates the idea of sparse PCA.
	More importantly, from a theoretical perspective as given in 
	\cite{Johnstone2009consistency,Zou2018}, sparse PCA is able to remediate some 
	inconsistency phenomenon present in the classical PCA under the high-dimensional setting.
	As a dimension reduction and feature extraction method, 
	sparse PCA can be widely applied to application areas where PCA is normally used;
	such as medical imaging \cite{Sjostrand2007sparse}, 
	biological knowledge mining \cite{Wu2021clusterProfiler},
	ecology study \cite{Gravuer2008strong}, cancer analysis \cite{Chen2013biclustering}, 
	neuroscience research \cite{Baden2016functional}, to mention a few examples.
	
	Let $A = [a_1, a_2, \dotsc, a_m]$ be an $n \times m$ data matrix, 
	where $n$ and  $m$ denote the numbers of features and samples, respectively.
	Without loss of generality, we assume that each row of $A$ has a zero mean. 
	Mathematically speaking, PCA finds an orthonormal basis 
	$Z \in \stiefel := \{Z \in \Rnp \mid Z\zz Z = I_p\}$ of a $p$-dimensional subspace
	such that the projection of samples $a_1, a_2, \dotsc, a_m$ on this subspace
	has the most variance.
	The set $\stiefel$ is usually referred to as the Stiefel manifold \cite{Stiefel1935}.
	Then PCA can be formally expressed as 
	the following optimization problem with the orthogonality constraints.
	\begin{equation}
		\label{eq:opt-pca}
		\begin{aligned}
			\min\limits_{Z \in \Rnp} \hspace{2mm} 
			& f (Z) := -\dfrac{1}{2} \tr \dkh{ Z\zz AA\zz Z } \\
			\st \hspace{3mm} & Z \in \stiefel,
		\end{aligned}
	\end{equation}
	where $\tr (\cdot)$ represents the trace of a given square matrix,
	and the column of $Z$ are called loading vectors or simply loadings.
	
	In the projected data $Z\zz A \in \bR^{p\times m}$, 
	the number of features is reduced from $n$ to $p$ 
	and each feature (row of $Z\zz A$) is a linear combination 
	of the original features (rows of $A$) with coefficients from $Z$.  
	For a sufficiently sparse $Z$, each reduced feature depends only on 
	a few original features instead of all of them, 
	leading to better interpretability in many applications.
	For this purpose, we consider the following 
	$\ell_1$-regularized optimization model for sparse PCA:
	\begin{equation}
		\label{eq:opt-spca-l1}
		\begin{aligned}
			\min\limits_{Z \in \Rnp} \hspace{2mm} 
			& \bar{f} (Z) := f (Z) + r (Z) \\
			\st \hspace{3mm} & Z \in \stiefel,
		\end{aligned}
	\end{equation}
	where the (non-operator) $\ell_1$-norm regularizer
	$r (Z) := \mu \norm{Z}_1 = \mu \sum_{i, j} \vert [Z]_{ij} \vert$
	is imposed to promote sparsity in $Z$,
	and $\mu > 0$ is the parameter used to control the amount of sparseness.
	Here, $[Z]_{ij}$ denotes the $(i, j)$-th entry of the matrix $Z$.
	Our distributed approach can efficiently tackle \eqref{eq:opt-spca-l1} with $p > 1$
	by pursuing sparsity and orthogonality at the same time.
	
	The optimization problem \eqref{eq:opt-spca-l1} is a penalized version 
	of the SCoTLASS model proposed in \cite{Jolliffe2003}.
	Evidently, there is a significant difficulty gap in going from the standard PCA 
	to sparse PCA in terms of both problem complexity and solution methodology.
	The standard PCA is polynomial-time solvable, 
	while the sparse PCA is NP-hard 
	if the $\ell_0$-regularizer is used to enforce sparsity \cite{Magdon2017np},
	though it is still unclear what is the computational time complexity of 
	solving the $\ell_1$-regularized model \eqref{eq:opt-spca-l1}.
	In this paper we will be content with a theoretical result of
	convergence to first-order stationary points of \eqref{eq:opt-spca-l1}.
	There are other formulations for sparse PCA,
	such as regression model \cite{Zou2006},
	semidefinite programming \cite{dAspremont2008,dAspremont2007},
	and matrix decompositions \cite{Shen2008,Witten2009}.
	A comparative study of these formulations is beyond the scope of this paper.
	We refer interested readers to \cite{Zou2018} for a recent survey of
	theoretical and computational results for sparse PCA.
	
\subsection{Distributed setting}

	We consider the following distributed setting.
	The data matrix $A$ is divided into $d$ blocks, each containing a number of samples; 
	namely, $A = [A_1, A_2, \dotsc, A_d]$  
	where $A_i \in \bR^{n\times m_i}$ so that $m_1 + \dotsb + m_d = m$.  
	This is a natural setting since each sample is a column of $A$.
	These submatrices $A_i$, $\iid$, are stored locally in $d$ locations,  
	possibly having been collected at different locations by different agents,
	and all the agents are connected through a communication network.
	According to the splitting scheme of \(A\),
	the function \(f (Z)\) can also be distributed into \(d\) agents, namely,
	\begin{equation}\label{def:sum_fi}
		f (Z) = \sumiid f_i (Z) 
		\mbox{~~with~~} 
		f_i (Z) = -\dfrac{1}{2} \tr \dkh{ Z\zz A_iA_i\zz Z }.
	\end{equation}

	In terms of network topology, we only need to assume that the network allows global summation 
	operations (say, through all-reduce type of communications \cite{Pacheco2011})  which are 
	required by our algorithm.  In particular, our algorithm will operate well in the federated-learning 
	setting~\cite{Mcmahan2017communication} where all the agents are connected to 
	a center server so that global summations can be readily achieved in the network.
	To this extent, we say our algorithm is a centralized algorithm.
	
	For most distributed algorithms, since the amount of communications per iteration
	remains essentially the same, the total communication overhead is proportional 
	to the required number of iterations regardless of the underlying network topology.
	Therefore, our goal is to devise an algorithm that converges fast in terms of iteration counts, 
	while considerations on other network topologies and communication patterns are 
	beyond the scope of the current paper.

	In certain applications, 
	such as those in healthcare and financial industry  \cite{Lou2017,Zhang2018a}, 
	preserving privacy of local data is of primary importance.  
	In this paper, we consider the following scenario: 
	each agent $i$ wants to keep its local dataset (i.e., $A_iA_i\zz$) from being discovered 
	by any other agents including the center.  
	In this situation, it is not an option to implement a pre-agreed encryption 
	or a coordinated masking operation.
	For convenience, we will call such a privacy situation of {\em intrinsic privacy}.  
	For an algorithm to preserve intrinsic privacy, 
	publicly exchanged information must be safe 
	so that none of the local-data matrices can be figured out by any means. 
	We will show later that, in general, algorithms based on traditional methods 
	with distributed matrix multiplications are not intrinsically private.

\subsection{Related works}		

	Optimization problems with orthogonality constraints 
	have been actively investigated in recent decades, for which
	many algorithms and solvers have been developed, such as,  
	gradient approaches \cite{Manton2002,Nishimori2005,Abrudan2008,Absil2008},
	conjugate gradient approaches \cite{Edelman1998,Sato2016dai,Zhu2017riemannian},
	constraint preserving updating schemes \cite{Wen2013,Jiang2015},
	Newton methods \cite{Hu2018,Hu2019},
	trust-region methods \cite{Absil2006},
	multipliers correction frameworks \cite{Gao2018,Wang2021multipliers},
	and orthonormalization-free approaches \cite{Gao2019,Xiao2020}.
	These aforementioned algorithms are designed for smooth objective functions, and
	are generally not suitable for problem \eqref{eq:opt-spca-l1}.
	
	There exist some algorithms specifically designed for 
	solving non-smooth optimization problems with orthogonality constraints;
	most of which adopt certain non-smooth optimization techniques 
	to the Stiefel manifold; for instances,
	Riemannian subgradient methods \cite{Ferreira1998,Ferreira2019}, 
	proximal point algorithms \cite{Bacak2016}, 
	non-smooth trust-region methods \cite{Grohs2016}, 
	gradient sampling methods \cite{Hosseini2017},
	and proximal gradient methods \cite{Chen2020,Huang2021riemannian}.
	Out of these algorithms, proximal gradient methods are among the most efficient.
	Specifically, Chen et al. \cite{Chen2020} propose a manifold proximal gradient method (ManPG) 
	and its accelerated version ManPG-Ada for non-smooth optimization problems over the Stiefel manifold.
	Starting from a current iterate, say $Z^{(k)} \in \stiefel$,
	ManPG and ManPG-Ada generate the next iterate by solving the following 
	subproblem restricted to the tangent space
	$\tanst{Z^{(k)}} = \{ D \in \Rnp \mid D\zz Z^{(k)} + (Z^{(k)})\zz D = 0  \}$:
	\begin{equation}\label{eq:manpg-sub}
		\min\limits_{D \in \tanst{Z^{(k)}}} \hspace{2mm} 
		- \jkh{ \sumiid A_i A_i\zz Z^{(k)}, D } + \dfrac{1}{2\eta} \norm{D}\fs + r(Z^{(k)} + D),
	\end{equation}
	which consumes the most computational time in the algorithm.
	In ManPG~\cite{Chen2020}, the semi-smooth Newton (SSN) method \cite{Xiao2018regularized} is 
	deployed to solve the above subproblem, and global convergence to stationary points is established
	with the convergence rate $O (1 / \sqrt{k})$.
	Huang and Wei \cite{Huang2021riemannian} later extend the framework of ManPG to 
	general Riemannian manifolds beyond the Stiefel manifold.
	Another class of algorithms first introduces auxiliary variables
	to split the objective function and the orthogonality constraint
	and then utilizes alternating minimization techniques to solve the resulting model,
	including SOC \cite{Lai2014}, MADMM \cite{Kovnatsky2016madmm}, and PAMAL \cite{Chen2016}.
	It is worth mentioning that SOC and MADMM lack a convergence guarantee.
	As for PAMAL,  although convergence is guaranteed,
	its numerical performance is heavily dependent on its parameters
	as discussed in \cite{Chen2020}.
	
	In principle, all the aforementioned algorithms can be adapted to 
	the distributed setting considered in this paper. 
	We take the algorithm ManPG as an example.
	Under the distributed setting, each agent computes 
	the local product $A_i A_i\zz Z^{(k)}$ individually,
	then one round of communications will gather the global sum 
	$\sum_{i = 1}^d A_i A_i\zz Z^{(k)}$ at a center.   
	The center then solves subproblem \eqref{eq:manpg-sub}
	and scatters the updated $Z^{(k+1)}$ back to all agents.
	At each iteration, distributed computation is basically limited to 
	the matrix-multiplication level.
	
	We point out that, without prior data-masking, distributed algorithms at the 
	matrix-multiplication level cannot preserve local-data privacy intrinsically.  
	Specifically, local data $A_i A_i\zz$, privately owned by agent $i$, can be uncovered by 
	anyone who has access to publicly exchanged products $A_iA_i\zz Z^{(k)}$,
	after collecting enough such products and then solving a system of linear equations 
	for the ``unknown'' $A_iA_i\zz$.  This idea works even for more complicated 
	iterative procedures as long as a mapping from data $A_iA_i\zz$ to a publicly 
	exchanged quantity is deterministic and known.
	
	To illustrate this data leakage situation, we now present a simple experiment 
	on algorithm ManPG-Ada \cite{Chen2020}.
	The test environment well be described in Section \ref{sec:numerical-result}.
	As mentioned earlier, at iteration $k$ the publicly shared quantity in 
	ManPG-Ada by agent $i$ is $S_i^{(k)} := A_i A_i\zz Z^{(k)}$,
	where $Z^{(k)}$ is the $k$-th iterate of ManPG-Ada accessible to all agents. 
	For instance, to recover the unknown local data $A_1A_1\zz$ 
	we construct the following linear system with unknown $Y \in\Rnn$:
	\begin{equation}\label{eq:SLRP}
		Y[Z^{(1)}, \cdots, Z^{(k)}] = [S_1^{(1)}, \cdots, S_1^{(k)}].
	\end{equation}
	Let $Y^{(k)}$ be the minimum norm least-squares solution 
	to the linear equation \eqref{eq:SLRP} at iteration $k$.
	We perform an experiment to illustrate that $Y^{(k)}$ 
	will quickly converge to $A_1A_1\zz$,
	when ManPG-Ada is deployed to solve the sparse PCA problem 
	on a randomly generated test matrix 
	with $n = 100$ and $m = 1280$
	(other parameters are $d = 10$, $p = 10$, and $\mu = 0.05$).
	Figure \ref{fig:security} depicts how the reconstruction error 
	$\|Y^{(k)} - A_1 A_1^{\top}\|_{\mathrm{F}}$ 
	and the stationarity violation\footnote{
	Suppose $D^{(k)}$ is the solution to  \eqref{eq:manpg-sub}.
	According to Lemma 5.3 in \cite{Chen2020},
	$X^{(k)}$ is a first-order stationary point if $D^{(k)} = 0$.
	Therefore, the stationarity violation is defined as $\| D^{(k)} \|\ff$.}
	vary, on a logarithmic scale, as the number of iterations increases.
	We observe that local data $A_1A_1\zz$ is obtained with high accuracy
	much faster than solving the sparse PCA problem.
	
	\begin{figure}[h]
		\centering
		\includegraphics[width=0.5\linewidth]{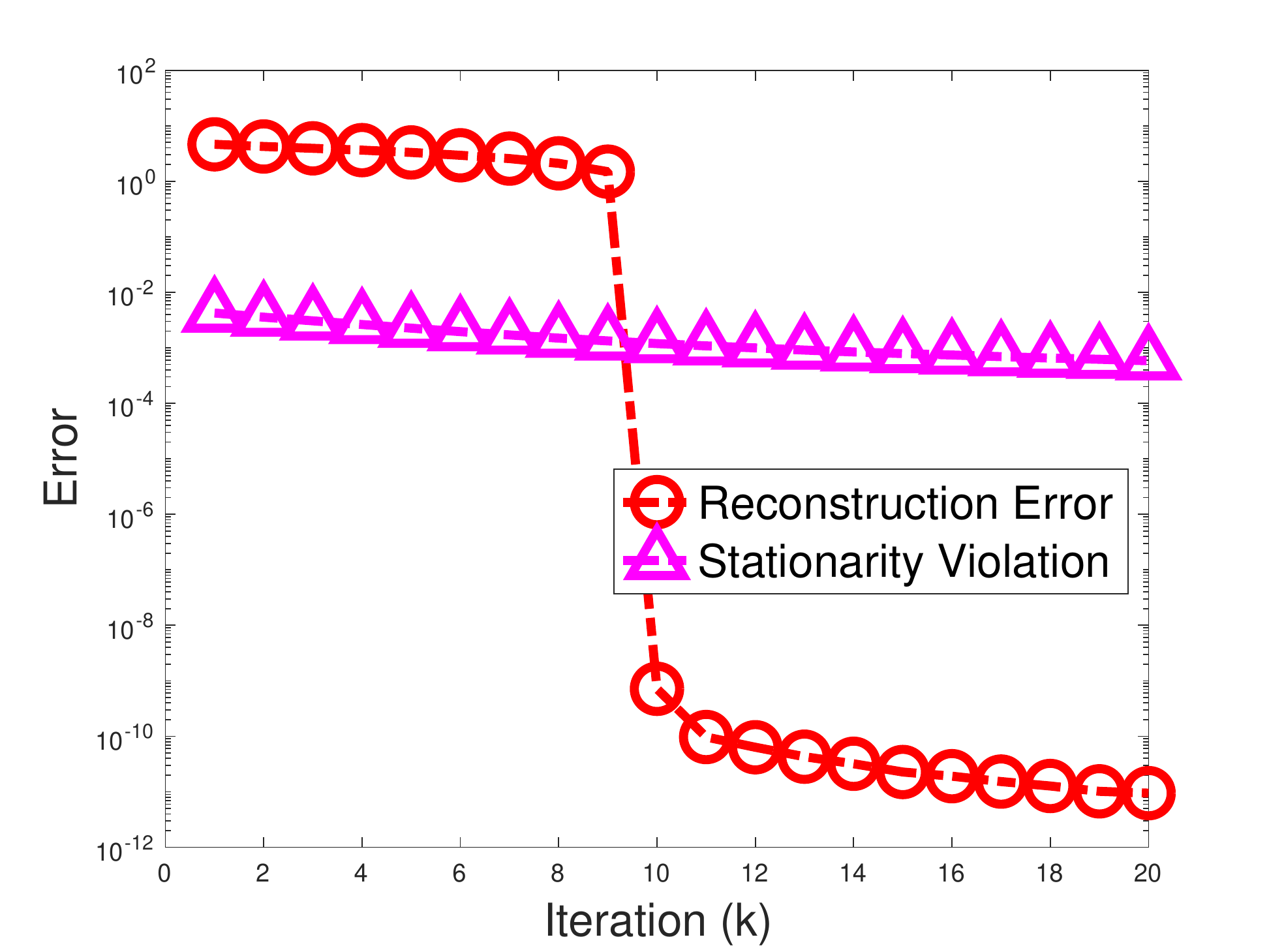}
		\caption{Local data uncovered by solving linear systems
			during ManPG-Ada iterations.}
		\label{fig:security}
	\end{figure}
	
	In order to handle distributed datasets, 
	some distributed ADMM-type algorithms have been introduced to 
	PCA and related problems.
	These methods achieve algorithm-level parallelization, which are generally more secure 
	than those based on matrix-multiplication-level parallelization.
	For instance, \cite{Hajinezhad2015} proposes a distributed algorithm for sparse PCA 
	with convergence to stationary points,
	but only studies the special case of $p = 1$.
	Moreover, under the distributed setting, \cite{Wang2020distributed} develop a subspace splitting strategy 
	to tackle the smooth optimization problem over the Grassmann manifold 
	without the $\ell_1$-norm regularizer in \eqref{eq:opt-spca-l1}.
	
	Recently, decentralized optimization has attracted attentions 
	partly due to its wide applications in wireless sensor networks.
	Only local communications are carried out in decentralized optimization,
	namely, agents exchange information only with their immediate neighbors.
	This is quite different from the distributed setting considered in the current paper.
	We refer interested readers to the references \cite{Gemp2020eigengame,Gang2021linearly,Gang2021fast,Andrade2021distributed,Ye2021,Chen2021decentralized,Wang2022decentralized}
	for some decentralized PCA algorithms.

\subsection{Main contributions}	

	Recently, a subspace splitting strategy has been proposed~\cite{Wang2020distributed} 
	to accelerate convergence of distributed algorithms for optimization problems over 
	Grassmann manifolds where objective functions are 
	orthogonal-transformation invariant\footnote{A function $f(X)$ is called 
	orthogonal-transformation invariant if $f (XO) = f(X)$ 
	for any $X \in \stiefel$ and $O \in \mathcal{S}_{p, p}$.} 
	and smooth.  In regularized problems such as sparse PCA, 
	orthogonal-transformation invariance and smoothness no longer hold.
	In this paper, we present a non-trivial extension of the
	subspace splitting strategy to more general optimization problems over Stiefel manifolds.
	{In particular, by incorporating the subspace splitting strategy into the framework of ManPG \cite{Chen2020}, 
	we have developed a distributed algorithm DSSAL1
	to solve the $\ell_1$-regularized optimization problem \eqref{eq:opt-spca-l1} for sparse PCA.}
	
	The main step of DSSAL1 is to solve the subproblem: 
	\begin{equation} \label{eq:DSSAL1-sub}
		\min\limits_{D \in \tanst{Z^{(k)}}} \hspace{2mm} 
		\jkh{ \sumiid Q_i^{(k)} Z^{(k)}, D } + \dfrac{1}{2\eta} \norm{D}\fs + r(Z^{(k)} + D),
	\end{equation}
	which is identical to subproblem~\eqref{eq:manpg-sub} in ManPG except that each local data matrix
	$A_iA_i\zz$ is replaced, or masked, by a matrix $-Q_i^{(k)}$ whose expression will be derived later. 
	
	In a sense, the main contributions of this paper can be attributed to this remarkably simple 
	replacement or masking, which brings two great benefits.
	Firstly, convergence will be significantly accelerated which will be shown
	empirically in Section \ref{sec:numerical-result}.
 	Secondly, publicly exchanged products $Q_i^{(k)}Z^{(k)}$ are no longer images of some 
	fixed and known mapping of $A_iA_i\zz$, making it practically impossible to uncover $A_iA_i\zz$ 
	from such products via equation-solving.

	Several innovations have been made in the development of our algorithms. 
	Firstly, in our algorithm, only the global variable can possibly converge to a solution, 
	while the local variables will never do.  The role of the local variables, 
	which are generally dense, is to help identify the right subspace.   Secondly, 
	we devise an inexact-solution strategy to effectively handle the difficult subproblem 
	for the global variable where orthogonality and sparsity are pursued simultaneously.
	Thirdly, we establish the global convergence to stationarity for our algorithm, 
	overcoming a number of technical difficulties associated with the rather complex 
	algorithm construction, as well as the non-convexity and non-smoothness in 
	problem \eqref{eq:opt-spca-l1}.

\subsection{Notations}

	The Euclidean inner product of two matrices $Y_1, Y_2$ 
	of the same size is defined as $\jkh{Y_1, Y_2}=\tr(Y_1\zz Y_2)$.
	The Frobenius norm and spectral norm of a given matrix $C$ 
	are denoted by $\norm{C}\ff$ and $\norm{C}_2$, respectively. 
	Given a differentiable function $g(X) : \Rnp \to \bR$, 
	the gradient of $g$ with respect to $X$ is denoted by $\nabla g(X)$.
	And the subdifferential of a Lipschitz continuous function $h(X)$ 
	is denoted by  $\partial h(X)$.
	The tangent space to the Stiefel manifold $\stiefel$ at $Z \in \stiefel$, 
	is represented by $\tanst{Z} = \{ Y \in \Rnp \mid Y\zz Z + Z\zz Y = 0  \}$, 
	and the orthogonal projection of $Y$ onto $\mathcal{T}_Z \stiefel$ is denoted by 
	$\projts{Z}{Y} = \dkh{I_n  -  ZZ\zz}Y + Z\dkh{Z\zz Y - Y \zz Z}/2$.
	For $X \in \stiefel$, we define $\Pv{X} := I_n - XX\zz$ standing for 
	the projection operator onto the null space of $X\zz$.
	Further notation will be introduced as it occurs.
	
\subsection{Organization}	

	The rest of this paper is organized as follows. 
	In Section \ref{sec:model}, we introduce a novel subspace-splitting model for sparse PCA,
	and investigate the structure of associated Lagrangian multipliers.
	Then a distributed algorithm is proposed to solve this model
	based on an ADMM-like framework in Section~\ref{sec:algorithmic-framework}.
	Moreover, convergence properties of the proposed algorithm are studied in Section \ref{sec:convergence-analysis}.
	Numerical experiments on a variety of test problems are presented in Section \ref{sec:numerical-result} 
	to evaluate the performance of the proposed algorithm.
	We draw final conclusions and discuss some possible future developments in the last section.
	
\section{Subspace-splitting model for sparse PCA}

\label{sec:model}

	We consider the scenario that the $i$-th component function $f_i (X)$ of $f$ in \eqref{def:sum_fi}
	can be evaluated only by the $i$-th agent since local dadaset $A_i$ is accessible only to the $i$-th agent. 
	In order to devise a distributed algorithm, the classic variable-splitting approach is to introduce 
	a set of local variables $\{X_i\}$ to make the sum of local functions nominally separable.
	Then a (centralized) distributed algorithm would maintain a global variable $Z$ 
	and impose variable-consensus constraints $X_i = Z$.
	
	Despite the regularizer term $r$ in \eqref{eq:opt-spca-l1}, all component functions $f_i(X)$ in 
	$f$ are still invariant under orthogonal transformations.  It should be natural for us to adapt the 
	subspace-splitting idea introduced in \cite{Wang2020distributed}, that is, to use
	the subspace-consensus constraints $X_i X_i\zz = Z Z\zz$ to accelerate convergence.
	In this paper, we propose to solve the following optimization problem:
	\begin{subequations}
		\label{eq:opt-ps}
		\begin{align}
			\label{eq:obj-ps}
			\min\limits_{ X_i, Z\in\Rnp } \hspace{2mm} & \sumiid f_i(X_i) + r(Z) \\
			\label{eq:con-orth-x}
			\st \hspace{11mm} & X_i\zz X_i = I_p,\hspace{8mm} \iid, \\
			\label{eq:con-subspace}
			& X_iX_i\zz = ZZ\zz, \hspace{3.5mm} \iid, \\
			\label{eq:con-orth-z}
			& Z\zz Z = I_p,
		\end{align}
	\end{subequations}
	which we will call the \textit{subspace-splitting} model for problem \eqref{eq:opt-spca-l1}, noting
	that both sides of \eqref{eq:con-subspace} are orthogonal projections onto subspaces.
	For brevity, we collect the global variable $Z$ and all local variables $\{X_i\}$ into a point $\dkh{ Z, \{X_i\} }$.
	A point $\dkh{ Z, \{X_i\} }$ is feasible if it satisfies the constraints \eqref{eq:con-orth-x}-\eqref{eq:con-orth-z}.
	

\subsection{Stationarity conditions}
	
	In this subsection, we aim to present the stationarity conditions 
	of the sparse PCA problem \eqref{eq:opt-spca-l1}.
	We first introduce the definition of Clarke subgradient \cite{Clarke1990} 
	for non-smooth functions.
	
	\begin{definition}
		Suppose $f: \Rnp \to \bR$ is a Lipschitz continuous function.
		The generalized directional derivative of $f$ at the point $X \in \Rnp$
		along the direction $H \in  \Rnp$ is defined by:
		\begin{equation*}
			f^{\circ} (X; H) := \limsup\limits_{Y \to X,\, t \to 0^+} \dfrac{f (Y + t H) - f(Y)}{t}.
		\end{equation*}
		Based on generalized directional derivative of $f$,
		the (Clark) subgradient of $f$ is defined by:
		\begin{equation*}
			\partial f(X) := \{G \in \Rnp \mid \jkh{G, H} \leq f^{\circ} (X; H) \}.
		\end{equation*}
	\end{definition}
	
	As discussed in \cite{Yang2014,Chen2020},
	the first-order stationarity condition of \eqref{eq:opt-spca-l1} can be stated as:
	\begin{equation*}\label{eq:kkt-ps}
		0 \in \projts{Z}{-A A\zz Z + \partial r(Z)}.
	\end{equation*}
	We provide an equivalent description of the above first-order stationarity condition,
	which will be used in the theoretical analysis.
	
	\begin{lemma}\label{le:kkt}
		A point $Z \in \stiefel$ is a first-order stationary point of \eqref{eq:opt-spca-l1}
		if and only if there exists $R(Z) \in \partial r(Z)$ such that the following conditions hold:
		\begin{equation}\label{eq:kkt}
			\left\{
			\begin{aligned}
				\Pv{Z} \dkh{- A A\zz Z + R(Z)} = 0, \\
				Z\zz R(Z) - R(Z)\zz Z = 0.
			\end{aligned}
			\right.
		\end{equation}
	\end{lemma}
	
	The proof of Lemma \ref{le:kkt} is put into Appendix \ref{apx:first-order}.
	Then we can characterize the first-order stationary points of \eqref{eq:opt-ps} 
	in the following manner.
	
	\begin{definition}\label{def:kkt-ps}
		Suppose $X_i\in\Rnp(\iid)$ and $Z\in\Rnp$.
		A point $(Z, \{X_i\})$ is called a first-order stationary point of \eqref{eq:opt-ps}
		if it is feasible and $Z$ satisfies the conditions in \eqref{eq:kkt}.
	\end{definition}

\subsection{Existence of low-rank multipliers}

\label{subsec:ex-low}

	By associating dual variables $\Gamma_i$, $\Lambda_i$, and $\Theta$ 
	to the constraints \eqref{eq:con-orth-x}, \eqref{eq:con-subspace}, 
	and \eqref{eq:con-orth-z}, respectively,
	we derive an equivalent description of first-order stationarity conditions of \eqref{eq:opt-ps}.
	
	\begin{proposition}\label{prop:kkt-multipliers}
		Suppose $X_i\in\Rnp(\iid)$ and $Z\in\Rnp$.
		A point $(Z, \{X_i\})$ is a first-order stationary point of \eqref{eq:opt-ps}
		if and only if there exist symmetric matrices 
		$\Lambda_i\in\Rnn$, $\Gamma_i\in\Rpp$, and $\Theta\in\Rpp$
		such that $(Z, \{X_i\})$ satisfies the following condition:
		\begin{equation}\label{eq:kkt-multipliers}
			\left\{
			\begin{aligned}
				0 & = A_i A_i\zz X_i + X_i \Gamma_i + \Lambda_i X_i, & \iid, \\
				0 & \in \partial r(Z) + \sumiid \Lambda_i Z - Z \Theta, \\
				0 & = X_i\zz X_i-I_p, &\iid, \\
				0 & = X_iX_i\zz-ZZ\zz, & \iid, \\
				0 & = Z\zz Z - I_p.
			\end{aligned}
			\right.
		\end{equation}
	\end{proposition}

	The proof of Proposition \ref{prop:kkt-multipliers} is relegated to Appendix \ref{apx:kkt-ps}.
	Actually, the equations in \eqref{eq:kkt-multipliers} can be viewed as 
	the KKT conditions of \eqref{eq:opt-ps},
	while $\Gamma_i\in\Rpp$, $\Lambda_i\in\Rnn$, and $\Theta\in\Rpp$ 
	are the Lagrangian multipliers corresponding to the constraints 
	\eqref{eq:con-orth-x}, \eqref{eq:con-subspace}, and \eqref{eq:con-orth-z}, respectively.

	In \cite{Wang2020distributed}, a low-rank and closed-form multiplier is devised
	with respect to the subspace constraint \eqref{eq:con-subspace} for the PCA problems,
	which can be expressed as:
	\begin{equation}\label{eq:multiplier}
		\Lambda_i = - X_i X_i\zz A_i A_i\zz \Pv{X_i} - \Pv{X_i} A_i A_i\zz X_i X_i\zz \; (\iid).
	\end{equation}
	In Appendix \eqref{apx:kkt-ps},
	we further verify that the above formulation is also valid for the sparse PCA problems
	at any first-order stationary point $(Z, \{X_i\})$.
	In the next section, 
	we will use \eqref{eq:multiplier} to update the multipliers in our framework of ADMM.
	This strategy simultaneously saves computational costs and storage requirements.

\section{Algorithmic framework}

\label{sec:algorithmic-framework}

	Now we describe the proposed algorithm, based on an ADMM-like framework, 
	to solve the subspace-splitting model \eqref{eq:opt-ps}.
	Note that there are three constraints \eqref{eq:con-orth-x}-\eqref{eq:con-orth-z}. 
	We only penalize the subspace constraints \eqref{eq:con-subspace} to the objective function, 
	and obtain the corresponding augmented Lagrangian function:
	\begin{equation}
		\label{eq:lag-ps}
		\cL( Z, \{X_i\}, \{\Lambda_i\} ) = \sumiid \cL_i(Z, X_i, \Lambda_i) + r(Z), 
	\end{equation}
	where
	\begin{equation*}
		\begin{aligned}
			\cL_i (Z, X_i, \Lambda_i) 
			= {} & - \dfrac{1}{2} \tr \dkh{ X_i\zz A_i A_i\zz X_i} 
			- \dfrac{1}{2} \jkh{ \Lambda_i, X_iX_i\zz - ZZ\zz } \\
			& + \dfrac{\beta_i}{4} \norm{ X_iX_i\zz - ZZ\zz }\fs,
		\end{aligned}
	\end{equation*}
	and $\beta_i > 0 (\iid)$ are penalty parameters.
	Schematically, we will follow the ADMM-like framework below to build an algorithm for solving
	subspace-splitting model \eqref{eq:opt-ps}, though we quickly add that the two optimization subproblems 
	below in \eqref{eq:sub-z} and \eqref{eq:sub-x} are not ``solved" in a normal sense since we may stop after
	a single iteration of an iterative scheme.
	\begin{numcases}{}
		\label{eq:sub-z}
		Z^{(k+1)} \approx \argmin_{Z\in\stiefel} \;
		\cL ( Z, \{X_i^{(k)}\}, \{\Lambda_i^{(k)}\} ). \\
		\label{eq:sub-x}
		X_i^{(k+1)} \approx \argmin_{X_i \in \stiefel} \;
		\cL_i (Z^{(k+1)}, X_i, \Lambda_i^{(k)}), \;\; \iid. \\
		\label{eq:Wi}
		W_i^{(k+1)} = - \Pv{X_i^{(k+1)}} A_i A_i\zz X_i^{(k+1)}, \;\; \iid. \\
		\label{eq:ps-mult}
		{\Lambda}_i^{(k+1)} = X_i^{(k+1)}(W_i^{(k+1)})\zz +  W_i^{(k+1)}(X_i^{(k+1)})\zz, \;\; \iid.
	\end{numcases}
	In the above framework, the superscript $(k)$ counts the number of iterations, 
	and the subscript $i$ indicates the number of agents.
	
	A novel feature in our algorithm is the way to update the multipliers $\Lambda_i$
	associated with the subspace constraints \eqref{eq:con-subspace}.
	Generally speaking, in the augment Lagrangian based approach, 
	the multipliers are updated by the dual ascent step
	\begin{equation*}
		\Lambda_i^{(k+1)} 
		= \Lambda_i^{(k)} - \tau \beta_i \dkh{X_i^{(k+1)}(X_i^{(k+1)})\zz - Z^{(k+1)}(Z^{(k+1)})\zz},
	\end{equation*}
	where $\tau > 0$ is the step size.  This standard method would require to store 
	and update an $n \times n$ matrix at each agent, which could be costly when $n$ is large.
	Instead, we use the low-rank updating formulas \eqref{eq:Wi}-\eqref{eq:ps-mult} based on
	the closed-form expression \eqref{eq:multiplier} derived in Section \ref{subsec:ex-low}. 
	In our iterative setting, these multiplier matrices are never stored but used in 
	matrix multiplications, in which the required additional storage for agent $i$ is
	for the $n \times p$ matrix $W_i$.
	
\subsection{Subproblem for global variable}

\label{subsec:sub-z}

	We now describe how to approximate subproblems \eqref{eq:sub-z} and \eqref{eq:sub-x}.
	By rearrangments, subproblem \eqref{eq:sub-z} reduces to:
	\begin{equation}\label{eq:ps-sub-z}
		\min\limits_{Z \in \stiefel} \hspace{2mm}  q^{(k)} (Z) := \dfrac{1}{2} \tr \dkh{ Z\zz Q^{(k)} Z} + r(Z)
	\end{equation}
	where $Q^{(k)} = \sum_{i = 1}^d Q_i^{(k)}$ and $Q_i^{(k)}$ is an $n \times n$ matrix defined by
	\begin{equation}
		Q_i^{(k)} =  \Lambda_i^{(k)} - \beta_i X_i^{(k)}(X_i^{(k)})\zz.
		\label{eq:Qik}
	\end{equation}
	We quickly point out here that it is not necessary to construct and store these $Q$-matrices explicitly 
	since we will only use them to multiply $n \times p$ matrices in an iterative scheme to obtain 
	approximate solutions to subproblem \eqref{eq:ps-sub-z}.  More details will follow later.

	Apparently, subproblem \eqref{eq:ps-sub-z} 
	pursues the orthogonality and sparsity simultaneously, and 
	is not easier to solve than the the original problem \eqref{eq:opt-spca-l1}.
	However, we will demonstrate later by both theoretical analysis and numerical 
	experiments that inexactly solving \eqref{eq:ps-sub-z} by conducting
	one proximal gradient step 
	is adequate for the global convergence.
	
	Starting from the current iterate $Z^{(k)}$,
	we first find a decent direction $D^{(k)}$ restricted to the tangent space $\tanst{Z^{(k)}}$ 
	by solving the following subproblem
	\begin{equation}\label{eq:ps-sub-z-pg}
		\begin{aligned}
			\min\limits_{D \in \Rnp} \hspace{2mm} 
			& \jkh{ Q^{(k)} Z^{(k)}, D } + \dfrac{1}{2\eta} \norm{D}\fs + r(Z^{(k)} + D) \\
			\st\hspace{3mm} & D\zz Z^{(k)} + (Z^{(k)})\zz D = 0,
		\end{aligned}
	\end{equation}
	where $\eta > 0$ is the step size.
	Since $Z^{(k)} + D^{(k)}$ does not necessarily lie on the Stiefel manifold $\stiefel$,
	we then perform a projection to bring it back to $\stiefel$,
	which can be represented as:
	\begin{equation*}
		Z^{(k+1)} = \proj_{\stiefel} \dkh{Z^{(k)} + D^{(k)}}.
	\end{equation*}
	Here, the orthogonal projection of a matrix $C \in \Rnp$ onto $\stiefel$ 
	is denoted by $\proj_{\stiefel} \dkh{C} = U_C V_C\zz$,
	where $U_C\Sigma_CV_C\zz$ is the economic form of the singular value decomposition of $C$.
	
	\begin{remark}
		We note that the subproblems solved by ManPG \cite{Chen2020} 
		are identical in form to our subproblem \eqref{eq:ps-sub-z-pg} 
		but with $Q^{(k)}$ replaced by the data matrix $AA\zz$.  
		That is, ManPG applies manifold proximal gradient steps to a fixed problem, 
		while our algorithm computes steps of the same type using a sequence of matrices, 
		each being updated to incorporate latest information.  
		From this point of view, our algorithm can be interpreted as an acceleration scheme 
		to reduce the number of outer-iterations, thereby reducing communication overheads.
	\end{remark}
	
	Since these $n \times n$ matrices $Q_i^{(k)} (\iid)$ are distributively maintained in $d$ agents,
	each agent is not able to independently solve subproblem \eqref{eq:ps-sub-z-pg}.
	Fortunately, we only need to calculate 
	\begin{equation}
		\label{eq:QkZk}
		Q^{(k)} Z^{(k)} = \sumiid Q_i^{(k)} Z^{(k)},
	\end{equation} 
	to solve this subproblem.
	In the distributed setting, the right hand side of \eqref{eq:QkZk} 
	can be accomplished by calculating $Q_i^{(k)} Z^{(k)}$
	in each agent and then invoking the all-reduce type of communication,
	where each agent just shares one $n \times p$ matrix.
	If the all-reduce type of communication is realized by the butterfly algorithm \cite{Pacheco2011},
	the communication overhead per iteration is $O\dkh{np\log(d)}$. 
	Furthermore, each local product $Q_i^{(k)} Z^{(k)}$ can be computed from
	\begin{equation} \label{eq:QikZk}
		\begin{aligned}
			Q_i^{(k)} Z^{(k)} 
			= {} & \Lambda_i^{(k)} Z^{(k)} - \beta_i X_i^{(k)}(X_i^{(k)})\zz Z^{(k)} \\
			= {} & X_i^{(k)}(W_i^{(k)})\zz Z^{(k)} +  W_i^{(k)}(X_i^{(k)})\zz Z^{(k)} \\
			& - \beta_i X_i^{(k)}(X_i^{(k)})\zz Z^{(k)},
		\end{aligned}
	\end{equation}
	with a computational cost in the order of $O (np^2)$.   From the above formula, 
	one observes that the $n \times n$ matrices $Q_i$ need not be stored explicitly.


	Now we consider how to efficiently solve subproblem \eqref{eq:ps-sub-z-pg}.
	By associating a multiplier $\Upsilon \in \Rpp$ to the linear equality constraint, 
	the Lagrangian function of \eqref{eq:ps-sub-z-pg} can be written as:
	\begin{equation*}
		\begin{aligned}
			\mathfrak{L} (D, \Upsilon) 
			= {} & \jkh{ Q^{(k)} Z^{(k)}, D } + \dfrac{1}{2\eta} \norm{D}\fs + r(Z^{(k)} + D) \\
			& - \dfrac{1}{2} \jkh{\Upsilon, D\zz Z^{(k)} + (Z^{(k)})\zz D }.
		\end{aligned}
	\end{equation*}
	We choose to apply the Uzawa method \cite{Arrow1958} to solve subproblem \eqref{eq:ps-sub-z-pg}.
	At the $j$-th inner iteration, 
	we first minimize the above Lagrangian function with respect to $D$ for a fixed $\Upsilon=\Upsilon(j)$:
	\begin{equation}
		\label{eq:uzawa-primal}
		\begin{aligned}
			D(j+1) = {} & \argmin\limits_{D \in \Rnp} \mathfrak{L} ( D, \Upsilon(j) ) \\
			= {} & \prox{\eta r}{Z^{(k)} - \eta \dkh{Q^{(k)} Z^{(k)} - Z^{(k)} \Upsilon(j)}} - Z^{(k)},
		\end{aligned}
	\end{equation}
	where $D(j)$ and $\Upsilon(j)$ denote the $j$-th inner iterate of $D$ and $\Upsilon$, respectively.
	Here, we use $\prox{g}{X}$ to denote the proximal mapping 
	of a given function $g: \Rnp \to \bR$ at the point $X \in \Rnp$, 
	which is defined by:
	\begin{equation*}
		\prox{g}{X} = \argmin\limits_{Y \in \Rnp} \hspace{2mm} g(Y) + \dfrac{1}{2} \norm{Y - X}\fs.
	\end{equation*}
	For the $\ell_1$-norm regularizer term $r(X) = \mu \norm{X}_1$,
	the proximal mapping in \eqref{eq:uzawa-primal} admits a closed-form solution:
	\begin{equation*}
		\fkh{ \prox{\eta r}{X} }_{ij} = 
		\left\{
		\begin{array}{ll}
			\fkh{X}_{ij} - \eta \mu, & \mbox{~if~} \fkh{X}_{ij} > \eta \mu, \\
			0, & \mbox{~if~} - \eta \mu \leq \fkh{X}_{ij} \leq \eta \mu,\\
			\fkh{X}_{ij} + \eta \mu, & \mbox{~if~} \fkh{X}_{ij} < - \eta \mu,
		\end{array}
		\right.
	\end{equation*}
	where the subscript $[\, \cdot \,]_{ij}$ represents the $(i, j)$-th entry of a matrix.
	Then the multiplier is updated by a dual ascent step:
	\begin{equation}
		\label{eq:uzawa-dual}
		\Upsilon(j+1) = \Upsilon(j) - \tau \dkh{D(j+1)\zz Z^{(k)} + (Z^{(k)})\zz D(j+1)},
	\end{equation}
	where $\tau > 0$ is the step size.
	These two steps are repeated until convergence.
	The complete framework is summarized in Algorithm \ref{alg:fixed-point}.

	The Uzawa method can be viewed as a special case of the primal-dual hybrid gradient algorithm (PDHG)
	developed in \cite{He2014} with the convergence rate $O(1 / k)$ in the ergodic sense under mild conditions.
	Moreover, as an inner solver it can bring a higher overall efficiency than the SSN method used by 
	ManPG \cite{Chen2020} (see \cite{Xiao2021penalty} for a recent study).

	\begin{algorithm2e}[h]
		\label{alg:fixed-point}
		\caption{Uzawa method for subproblem \eqref{eq:ps-sub-z-pg}.} 
		
		
		\KwIn{ $Z^{(k)}$, $Q^{(k)} Z^{(k)}$, and $\eta$ 
			in subproblem \eqref{eq:ps-sub-z-pg}.}
		
		Set $j := 0$, and choose the step size $\tau > 0$ as well as the initial variable $\Upsilon(0)$. 
		
		\While{not converged}
		{
			Compute $D(j+1)$ by \eqref{eq:uzawa-primal}.
			
			Update  $\Upsilon(j+1)$ by \eqref{eq:uzawa-dual}.
			
			Set $j := j + 1$. 
		}
		
		\KwOut{$D(j)$.}	
	\end{algorithm2e}

\subsection{Subproblems for local variables}

\label{subsec:sub-x}

	In this subsection, we focus on  for the $i$-th local variable $X_i$, 
	which can be rearranged as the following equivalent problem:
	\begin{equation}\label{eq:ps-sub-x}
		\min\limits_{X_i \in \stiefel} \hspace{2mm} h_i^{(k)} (X_i) 
		:= - \dfrac{1}{2} \tr \dkh{ X_i\zz H_i^{(k)} X_i }.
	\end{equation}
	Here, $H_i^{(k)}$ is an $n \times n$ real symmetric matrix:
	\begin{equation}\label{eq:Hik}
		H_i^{(k)} = A_i A_i\zz + \Lambda_i^{(k)} + \beta_i Z^{(k + 1)}(Z^{(k + 1)})\zz,
	\end{equation}
	which is only related to the local data $A_i$.
	This is a standard eigenvalue problem where one needs to compute
	a $p$-dimensional dominant eigenspace of $H_i^{(k)}$.
	
	As a subproblem, it is not necessary to solve \eqref{eq:ps-sub-x} to high precision.
	In practice, we just need to find a point $X_i^{(k+1)} \in \stiefel$
	satisfying the following two conditions,
	which suffices to be a good inexact solution empirically, 
	and to guarantee the global convergence of the whole algorithm.
	The first condition demands a sufficient decrease in function value:
	\begin{equation}
		\label{eq:ps-sub-x-con-1}
		h_i^{(k)} \left( X_i^{(k)} \right) - h_i^{(k)} \left( X_i^{(k+1)} \right) 
		\geq \dfrac{c_i}{c_i^{\prime}\norm{A_i}_2^2 + \beta_i} 
		\norm{ \Pv{X_i^{(k)}} H_i^{(k)} X_i^{(k)} }\fs,
	\end{equation}
	where $c_i > 0$ and $c_i^{\prime} > 0$ are two constants independent of $\beta_i$.
	The second condition is a sufficient decrease in KKT violation:
	\begin{equation}\label{eq:ps-sub-x-con-2}
		\norm{ \Pv{X_i^{(k + 1)}} H_i^{(k)} X_i^{(k+1)} }\ff 
		\leq \delta_i \norm{ \Pv{X_i^{(k)}} H_i^{(k)} X_i^{(k)} }\ff,
	\end{equation}
	where $\delta_i \in [ 0, 1 )$ is a constant independent of $\beta_i$.
	It turns out that these two rather weak termination conditions for subproblem~\eqref{eq:ps-sub-x}
	are sufficient for us to derive global convergence of our ADMM-like algorithm framework
	\eqref{eq:sub-z}-\eqref{eq:ps-mult}.

	In practice, we can combine a warm-start strategy with a single iteration of SSI \cite{Rutishauser1970} 
	to generate the next iterate $X_i^{(k + 1)}= \proj_{\stiefel} \dkh{H_i^{(k)} X_i^{(k)}}$, i.e.,
	\begin{equation} \label{eq:Xik+1}
		\begin{aligned}
			X_i^{(k + 1)} 
			& = \proj_{\stiefel} \dkh{A_i A_i\zz X_i^{(k)} + \Lambda_i^{(k)} X_i^{(k)}
			+ \beta_i Z^{(k + 1)}(Z^{(k + 1)})\zz X_i^{(k)}} \\
			& = \proj_{\stiefel} \dkh{A_i A_i\zz X_i^{(k)} + W_i^{(k)}
			+ \beta_i Z^{(k + 1)}(Z^{(k + 1)})\zz X_i^{(k)}},
		\end{aligned}
	\end{equation}
	which can be computed in the order of $O(np^2)$ floating-point operations,
	given that the term $A_i A_i\zz X_i^{(k)}$ is inherited from the last iteration as
	a result of updating $W_i^{(k)}$, see \eqref{eq:Wi}.

\subsection{Algorithm description}

\label{subsec:framework}

	We formally present the detailed algorithmic framework as Algorithm \ref{alg:DSSAL1} below,
	named \textit{\underline{d}istributed \underline{s}ubspace \underline{s}plitting 
	\underline{a}lgorithm with \underline{$\ell_1$} regularization} 
	and abbreviated to DSSAL1.
	In the distributed environment, 
	all agents are initiated from the same point $Z^{(0)} \in \stiefel$.
	And the initial guess of multipliers are computed by \eqref{eq:ps-mult}.
	After initialization, all agents first solve the common subproblem for $Z$ 
	collaboratively by certain communication strategy.
	Then each agent solves its subproblem for $X_i$ 
	and updates its multiplier $\Lambda_i$.
	These two steps only involve the local data privately stored at each agent,
	and hence can be carried out in $d$ agents concurrently.
	This procedure is repeated until convergence.

	\begin{algorithm2e}[h]
		\label{alg:DSSAL1}
		\caption{Distributed Subspace Splitting Algorithm with $\ell_1$ regularization (DSSAL1).} 
		
		
		\KwIn{functions $f_i(\iid)$ and $r$.}
		
		Set $k := 0$, choose penalty parameters $\hkh{ \beta_i }$, and initialize $Z^{(0)}$.
		
		Set $X_i^{(0)} = Z^{(0)}$ for $\iid$.
		
		Compute the initial multipliers $\{\Lambda_i^{(0)}\}$ by \eqref{eq:ps-mult}.
		
		\While{not converged} 
		{
			Solve  \eqref{eq:ps-sub-z-pg} to obtain $D^{(k)}$ by Algorithm \ref{alg:fixed-point}.
			
			Set $Z^{(k+1)} = \proj_{\stiefel}\dkh{Z^{(k)} + D^{(k)}}$.
			
			\For{$i=1,\dotsc,d$}
			{
				Find $X_i^{(k+1)} \in \stiefel$ satisfying \eqref{eq:ps-sub-x-con-1} and \eqref{eq:ps-sub-x-con-2}.
				
				Update the multipliers $\Lambda_i^{(k+1)}$ by \eqref{eq:ps-mult}.
			}
			
			
			Set $k := k + 1$.
		}
		
		\KwOut{$Z^{(k)}$.}
		
	\end{algorithm2e}

\subsection{Data privacy}

\label{subsec:privacy}
	
	We claim that DSSAL1 can naturally protect the intrinsic privacy of local data.
	To form the global sum in \eqref{eq:QkZk}, the shared information in DSSAL1 
	at iteration $k$ from the $i$-th agent is $S_i^{(k)} := Q_i^{(k)} Z^{(k)}$ where
	$Z^{(k)}$ is known to all agents.  However, the $n\times p$ system of equations,
	$Q_i^{(k)} Z^{(k)}=S_i^{(k)}$, is insufficient for obtaining the $n\times n$ mask matrix 
	$Q_i^{(k)}$ which changes from iteration to iteration.  Secondly, 
	even if a few mask matrices $Q_i^{(k)}$ were unveiled, 
	it would still be impossible to derive the local data matrix $A_i A_i\zz$ from these
	$Q_i^{(k)}$ without knowing corresponding $X_i^{(k)}$ (and $\beta_i$) which are 
	always kept privately by the $i$-th agent.  Finally, consider the ideal ``converged" 
	case where $X_iX_i\zz=ZZ\zz$ held at iteration $k$, and $\beta_i$ were known. 
	In this case,  $Q_i^{(k)}$ would be a known linear function of $A_iA_i\zz$
	parameterized by $Z^{(k)}$.  Still, the $n\times p$ system $Q_i^{(k)}Z^{(k)}=S_i^{(k)}$
	would not be sufficient to uncover the $n\times n$ local data matrix $A_iA_i\zz$
	(strictly speaking, one only needs to recover $n(n+1)/2$ entries since $A_iA_i\zz$ is symmetric).
	Based on this discussion, we call DSSAL1 a privacy-aware method.

\subsection{Computational cost}

	We conclude this section by discussing the computational cost of our algorithm 
	per iteration.  We first compute the matrix multiplication $Q^{(k)} Z^{(k)}$ 
	by \eqref{eq:QkZk} and \eqref{eq:QikZk},
	whose computational cost for each agent is $O (np^2)$ as mentioned earlier.
	Then, at the center, the Uzawa method is applied to solving subproblem 
	\eqref{eq:ps-sub-z-pg} which has a per-iteration complexity $O (np^2)$.
	In practice, it usually takes very few iterations to generate $Z^{(k + 1)}$.
	Next, each agent uses a single iteration of SSI to generate $X_i^{(k + 1)}$ by \eqref{eq:Xik+1}
	with the computational cost $O (np^2)$ as discussed before.
	Finally, agent $i$ updates $W_i^{(k+1)}$ by \eqref{eq:Wi}
	with the computational cost $4npm_i + O(n p^2)$ 
	(which represents the multiplier matrix $\Lambda_i^{(k+1)}$ implicitly).
	Overall, for each agent, the computational cost of our algorithm 
	is $4npm_i + O(n p^2)$ per iteration.  At the center, the computational cost 
	for approximately solving \eqref{eq:ps-sub-z-pg} is, empirically speaking, $O (n p^2)$.

\section{Convergence analysis}

\label{sec:convergence-analysis} 

	In this section, we analyze the global convergence of Algorithm \ref{alg:DSSAL1}
	and prove that a sequence $\{Z^{(k)}\}$ generated by Algorithm \ref{alg:DSSAL1} 
	has at least one accumulation point, and any accumulation point is a first-order 
	stationary point.  A global, sub-linear convergence rate is also established.
	
	We start with a property that a feasible point is first-order stationary 
	if no progress can be made by solving  \eqref{eq:ps-sub-z-pg}.

	\begin{lemma}
		\label{le:optimality}
		Let $(Z^{(k)}, \{X_i^{(k)}\})$ be feasible.  Then $Z^{(k)}$ is a first-order stationary point of 
		the sparse PCA problem \eqref{eq:opt-spca-l1} if $D^{(k)} = 0$ is the minimizer of 
		subproblem \eqref{eq:ps-sub-z-pg}. 
	\end{lemma}
	
	The proof of Lemma \ref{le:optimality} is deferred to Appendix \ref{apx:optimality}.
	It motivates the following definition of an  $\epsilon$-stationary point for
	the subspace-splitting model \eqref{eq:opt-ps}.
	
	\begin{definition}
		Suppose $Z^{(k)}$ is the $k$-th iterate of Algorithm \ref{alg:DSSAL1}.
		Then $Z^{(k)}$ is called an $\epsilon$-stationary point
		if the following condition holds:
		\begin{equation*}
			\dfrac{1}{d} \sumiid \norm{Z^{(k)}(Z^{(k)})\zz - X_i^{(k)}(X_i^{(k)})\zz}\fs 
			+ \norm{D^{(k)}}\fs \leq \epsilon^2,
		\end{equation*}
		where $\epsilon > 0$ is a small constant.
	\end{definition}
	
	In order to prove the convergence of our algorithm, we need to impose some mild conditions 
	on algorithm parameters, which are summarized below.
	
	\begin{condition}\label{asp:step}
		The algorithm parameters $\eta > 0$, $c_i > 0$, $c_i^{\prime} > 0$, $\delta_i \in [0, 1)$,
		as well as two auxiliary parameters $\rho \geq 1$ and $\underline{\sigma} \in (0, 1)$,
		satisfy the following conditions:
		\begin{equation*}
			0 < \eta < \dfrac{1}{2 \bar{M}}, \;
			0 \leq \delta_i < \dfrac{\underline{\sigma}}{2 \sqrt{\rho d}}, \;
			0 < \underline{\sigma} < \min\hkh{1, \dfrac{1}{\sqrt{c_i}}}, \;
			\iid,
		\end{equation*}
		where $\bar{M} = \mu \sqrt{np} / 2 + 2 \norm{A}\fs + \sqrt{p} \sum_{i = 1}^d \beta_i > 0$ is a constant.
	\end{condition}
	
	\begin{condition}
		\label{asp:beta}
		Each penalty parameter $\beta_i (\iid)$ in \eqref{eq:lag-ps} has a lower bound
		\begin{equation*}
			\begin{aligned}
				\max \left\{
				\xi_i \norm{A_i}_2^2, \;
				\dfrac{8 (\mu n p + \sqrt{p} \norm{A}\fs)}{(1 - \underline{\sigma}^2)}, \;
				\dfrac{ 6\sqrt{p} \norm{A_i}\fs}{c_i \underline{\sigma}^2 (1 - \underline{\sigma}^2)}
				\right\},
			\end{aligned}
		\end{equation*}
		where 
		$\xi_i = \max \{ c_i^{\prime}, \; 4\sqrt{2} / \underline{\sigma}, \; 4( 2 \sqrt{\rho d} + \sqrt{2} ) / (\underline{\sigma} - 2\sqrt{\rho d} \delta_i), \;
			4(\sqrt{2 \rho d} + 1) / (c_i \underline{\sigma}^2 \rho d) \} > 0$
		is a constant.
		In addition, $\beta_i\leq \rho\beta_j$ holds for any $i, j \in \{1, \dotsc, d\}$.
	\end{condition}

	We note that the above technical conditions are not necessary in a practical implementation,
	They are introduced purely for the purpose of theoretical analysis to facilitate obtaining a global 
	convergence rate and corresponding worst-case complexity for Algorithm \ref{alg:DSSAL1}.

	\begin{theorem}
		\label{thm:global}
		
		Suppose $\{Z^{(k)}\}$ is an iterate sequence generated by Algorithm \ref{alg:DSSAL1},
		starting from an arbitrary orthogonal matrix $Z^{(0)} \in \stiefel$, with parameters satisfying 
		Conditions \ref{asp:step} and \ref{asp:beta}. 
		Then $\{Z^{(k)}\}$ has at least one accumulation point and any accumulation point 
		must be a first-order stationary point of the sparse PCA problem \eqref{eq:opt-spca-l1}.
		Moreover, for any integer $K > 1$, it holds that
		\begin{equation*}
			\min\limits_{k = 1, \dotsc, K} 
			\hkh{ \norm{D^{(k)}}\fs + 
			\dfrac{1}{d} \sumiid \norm{Z^{(k)}(Z^{(k)})\zz - X_i^{(k)}(X_i^{(k)})\zz}\fs }
			\leq \dfrac{C}{K},
		\end{equation*}
		where $C > 0$ is a constant.
		
	\end{theorem}
	
	The proof of Theorem \ref{thm:global}, which is rather complicated and lengthy, will be given
	in Appendix \ref{apx:global}.  The global sub-linear convergence rate in Theorem \ref{thm:global}
	guarantees that DSSAL1 is able to return an $\epsilon$-stationary point in at most 
	$O(1/\epsilon^2)$ iterations.  
	Since DSSAL1 performs one round of communication per iteration, 
	the number of communication rounds required to obtain an $\epsilon$-stationary point 
	is also $O(1/\epsilon^2)$ at the most.

\section{Numerical results}

\label{sec:numerical-result}
	
	In this section, we evaluate the empirical performance of DSSAL1 
	through a set of comprehensive numerical experiments.
	All the experiments throughout this section are performed on a high-performance 
	computing cluster
	\footnote{More information at \url{http://lsec.cc.ac.cn/chinese/lsec/LSSC-IVintroduction.pdf}},
	called LSSC-IV	which is maintained at the State Key Laboratory of Scientific and 
	Engineering Computing (LSEC), Chinese Academy of Sciences.  The
	LSSC-IV cluster has 408 nodes, each consisting of two Inter(R) Xeon(R) Gold 6140 processors 
	(at $2.30$GHz $\times 18$) with $192$GB memory, running under  
	the operating system Red Hat Enterprise Linux Server 7.3.

	We compare the performance of DSSAL1 with two  state-of-the-art algorithms:
	(1) an ADMM-type algorithm called  SOC \cite{Lai2014}
	and (2) a manifold proximal gradient method called ManPG-Ada \cite{Chen2020}.
	Since open-source, parallel codes for the above two algorithms are not available, 
	to conduct experiments under the distributed environment of the LSSC-IV cluster, 
	we implemented the two existing algorithms and our own algorithm DSSAL1 
	in C++ with MPI for inter-process communications\footnote{
	Our code is downloadable from \url{http://lsec.cc.ac.cn/~liuxin/code.html}}.
	For the two existing algorithms, we set all parameters to their default values 
	as described in \cite{Lai2014, Chen2020}.  The linear algebra library Eigen\footnote
	{Available from \url{https://eigen.tuxfamily.org/index.php?title=Main_Page}} 
	(version 3.3.8) is adopted for matrix computation tasks.


\subsection{DSSAL1 Implementation details}

	
	In Algorithm \ref{alg:DSSAL1}, we set the penalty parameters to 
		$\beta_i = 0.1(\|\nabla f_i (X_i^{(0)})\|\ff + \mu)$, 
	and in subproblem \eqref{eq:ps-sub-z-pg} we set the hyperparameter 
	to $\eta = 1/(\sum_{i = 1}^d \beta_i)$ . 
	In Algorithm \ref{alg:fixed-point}, we set the step size to $\tau = 1/\dkh{2\eta}$ 
	and terminate the algorithm whenever
	\begin{equation*} 
		\norm{ D(j)\zz Z^{(k)} + (Z^{(k)})\zz D(j) }\ff \leq \norm{D^{(k-1)}}\ff
	\end{equation*}
	is satisfied or the number of iterations reaches $10$.
	
	We use the well-known SSI method \cite{Rutishauser1970}
	to obtain a very rough solution to subproblem~\eqref{eq:ps-sub-x} for $X_i$.
	More specifically, we initialize $X_i$ to the previous iterate $X_i^{(k)}$
	and perform a single SSI iteration, as given by \eqref{eq:Xik+1}, to generate 
	the next iterate $X_i^{(k+1)}$.
	
	The stopping criteria used in Algorithm \ref{alg:DSSAL1} are
	\begin{equation} \label{eq:stop}
		\dfrac{1}{d} \sumiid  \norm{Z^{(k)}(Z^{(k)})\zz - X_i^{(k)}(X_i^{(k)})\zz}\ff \leq \epsilon_c
		\mbox{~~and~~}
		\norm{D^{(k)}}\ff \leq \epsilon_g,
	\end{equation}
	where $\epsilon_c$ and $\epsilon_g$ are two small positive constants.
	Unless otherwise specified, $\epsilon_c$ and $\epsilon_g$ are set to $10^{-6}$ and $10^{-8}np$, respectively. 
	Algorithm \ref{alg:DSSAL1} is also terminated once the iteration count reaches $\mathtt{MaxIter}=50000$.

\subsection{Synthetic data generation}
	\label{sec-data}
 
	In our experiments, a synthetic data matrix $A \in \Rnm$ is constructed into the form of (economy-size) SVD:
	\begin{equation}
		\label{eq:gen-A}
		A = U \Sigma V\zz,
	\end{equation}
	where $U\in\Rnn$ and $V\in\bR^{m\times n}$ satisfy $U\zz U = V\zz V = I_n$
	and $\Sigma \in \Rnn$ is nonnegative and diagonal.  Specifically, $U$ and $V$
	are results of orthonormalization of random matrices whose entries are drawn 
	independently and identically from $[-1,1]$ under the uniform distribution, and
	\begin{equation*}
		\Sigma_{ii}=\xi^{1-i}, \quad \iin,
	\end{equation*}  
	where the parameter $\xi\geq1$ determines the decay rate of singular values.
	Finally, we apply the standard PCA pre-processing operations to a data matrix $A=U\Sigma V\zz$ 
	by subtracting the sample-mean from each sample and then normalizing the rows of the 
	resulting matrix to make them of unit $\ell_2$-norm.  For our synthetic data matrices, 
	such pre-processing will only slightly perturb the decay rate of the singular-values 
	before the pre-processing which is uniformly equal to $1/\xi$ by construction.
	Unless specified otherwise, the default value for the decay-rate parameter is $\xi = 1.1$.
	
	In the numerical experiments, all the algorithms are started from the same initial points.
	Since the optimization problem is non-convex, different solvers may still occasionally 
	return different solutions when starting from a common initial point at random.
	As suggested in \cite{Chen2020}, to increase the chance that all solvers find the same solution, 
	we first run the Riemannian subgradient method \cite{Bento2017,Ferreira2019} for $500$ iterations 
	and then use the resulting output as the common starting point.
	
\subsection{Comprehensive comparison on synthetic data}

	In order to do a thorough evaluation on the empirical performance of DSSAL1, 
	we design four groups of test data matrices, generated as in Subsection~\ref{sec-data}.
	In each group, there is only one parameter varying while all the others are fixed.
	Specifically, the varying and fixed parameters for the four groups are as follows:
	\begin{enumerate}
		
		\item varying sample dimension $n = 1000 + 500j$ for $j = 0, 1, 2, 3, 4$, while 
		$m = 128000$, $p = 10$, $\mu = 0.5$, $d = 128$;
		
		\item varying number of computed loading vectors $p = 10 + 5j$ for $j = 0, 1, 2, 3, 4$, while
		$n = 1000$, $m = 128000$, $\mu= 0.3$, $d = 128$;
		
		\item varying regularization parameter $\mu = 0.2 + 0.2j$ for $j = 0, 1, 2, 3, 4$, while 
		$n = 1000$, $m = 128000$, $p = 10$, $d = 128$;
		
		\item varying number of cores $d = 16 \times 2^{j}$ for $j = 0, 1, 2, 3, 4$, while
		$n = 1000$, $m = 256000$, $p = 10$, $\mu = 0.5$.
		
	\end{enumerate}
	
	Additional experimental results for varying $\xi$ will be presented in the next subsection.
	Numerical results obtained for the above four test groups are presented 
	in Tables \ref{tb:spca_n} to \ref{tb:spca_d}, respectively,
	where we record wall-clock times in seconds and total rounds of communication.
	The average function values and sparsity levels 
	for the four groups of tests are provided in Table \ref{tb:spca_aver}.
	When computing the sparsity of a solution matrix (i.e., the percentage of zero elements),
	we set a matrix element to zero when its absolute value is less than $10^{-5}$. 
	
	From Table \ref{tb:spca_aver}, we see that all three algorithms have attained comparable solution qualities 
	with similar function values and sparsity levels in the four groups of testing problems.

	\begin{table}[ht]
		\begin{center}
			\begin{minipage}{\textwidth}
				\caption{Comparison of DSSAL1, ManPG-Ada, and SOC for different $n$.}
				\label{tb:spca_n}
				\begin{tabular*}{\textwidth}{@{\extracolsep{\fill}}ccccccc@{\extracolsep{\fill}}}
					\toprule%
					& \multicolumn{3}{@{}c@{}}{Wall-clock time in seconds} 
					& \multicolumn{3}{@{}c@{}}{Rounds of communication} \\
					\cmidrule{2-4}\cmidrule{5-7}%
					$n$ & DSSAL1 & ManPG-Ada & SOC 
					& DSSAL1 & ManPG-Ada & SOC \\
					\midrule
					1000  & {\bf 10.97} & 26.76 & 72.07
					& {\bf 655}  & 1794 &  7569\\
					1500  & {\bf 6.61} & 16.12 & 57.75 
					& {\bf 223} & 642 & 2201 \\
					2000  & {\bf 49.22} & 172.32 & 725.63
					& {\bf 1238} & 5054 & 20444 \\
					2500  & {\bf 181.97} & 412.51 & 2296.26
					& {\bf 3753} & 10971 & 45707 \\
					3000  & {\bf 34.58} & 153.40 & 860.76 
					& {\bf 680} & 2789 & 11480 \\
					\bottomrule
				\end{tabular*}
			\end{minipage}
		\end{center}
	\end{table}

	\begin{table}[h]
		\begin{center}
			\begin{minipage}{\textwidth}
				\caption{Comparison of DSSAL1, ManPG-Ada, and SOC for different $p$.}
				\label{tb:spca_p}
				\begin{tabular*}{\textwidth}{@{\extracolsep{\fill}}ccccccc@{\extracolsep{\fill}}}
					\toprule%
					& \multicolumn{3}{@{}c@{}}{Wall-clock time in seconds} 
					& \multicolumn{3}{@{}c@{}}{Rounds of communication} \\
					\cmidrule{2-4}\cmidrule{5-7}%
					$p$ & DSSAL1 & ManPG-Ada & SOC 
					& DSSAL1 & ManPG-Ada & SOC \\
					\midrule
					10  & {\bf 9.74} & 26.42 & 75.25
					& {\bf 629}  & 1622 &  6652\\
					15  & {\bf 29.99} & 56.31 & 153.98
					& {\bf 1728} & 3586 & 15865 \\
					20  & {\bf 110.26} & 239.51 & 466.22
					& {\bf 6144} & 14107 & 44086 \\
					25  & {\bf 68.79} & 148.34 & 334.58
					& {\bf 3030} & 6153 & 27060 \\
					30  & {\bf 110.16} & 173.11 & 204.87
					& {\bf 5133} & 6966 & 14621 \\
					\bottomrule
				\end{tabular*}
			\end{minipage}
		\end{center}
	\end{table}

	\begin{table}[h]
		\begin{center}
			\begin{minipage}{\textwidth}
				\caption{Comparison of DSSAL1, ManPG-Ada, and SOC for different $\mu$.}
				\label{tb:spca_mu}
				\begin{tabular*}{\textwidth}{@{\extracolsep{\fill}}ccccccc@{\extracolsep{\fill}}}
					\toprule%
					& \multicolumn{3}{@{}c@{}}{Wall-clock time in seconds} 
					& \multicolumn{3}{@{}c@{}}{Rounds of communication} \\
					\cmidrule{2-4}\cmidrule{5-7}%
					$\mu$ & DSSAL1 & ManPG-Ada & SOC 
					& DSSAL1 & ManPG-Ada & SOC \\
					\midrule
					0.2  & {\bf 25.05} & 59.46 & 430.37
					& {\bf 1537}  & 4140 &  50000\\
					0.4  & {\bf 22.16} & 46.76 & 115.36
					& {\bf 1393} & 3061 & 13680 \\
					0.6  & {\bf 11.61} & 29.08 & 81.82
					& {\bf 838} & 1899 & 8278 \\
					0.8  & {\bf 10.09} & 40.07 & 66.46
					& {\bf 733} & 3369 & 8121 \\
					1.0  & {\bf 9.12} & 19.80 & 56.33
					& {\bf 655} & 1348 & 6396 \\
					\bottomrule
				\end{tabular*}
			\end{minipage}
		\end{center}
	\end{table}
	\begin{table}[hh]
		\begin{center}
			\begin{minipage}{\textwidth}
				\caption{Comparison of DSSAL1, ManPG-Ada, and SOC for different $d$.}
				\label{tb:spca_d}
				\begin{tabular*}{\textwidth}{@{\extracolsep{\fill}}ccccccc@{\extracolsep{\fill}}}
					\toprule%
					& \multicolumn{3}{@{}c@{}}{Wall-clock time in seconds} 
					& \multicolumn{3}{@{}c@{}}{Rounds of communication} \\
					\cmidrule{2-4}\cmidrule{5-7}%
					$d$ & DSSAL1 & ManPG-Ada & SOC 
					& DSSAL1 & ManPG-Ada & SOC \\
					\midrule
					16  & {\bf 161.72} & 168.36 & 383.14
					& {\bf 897}  & 1169 &  5175\\
					32  & {\bf 106.95} & 135.81 & 312.20
					& {\bf 808} & 1169 & 5175 \\
					64  & {\bf 54.33} & 68.89 & 167.27
					& {\bf 753} & 1169 & 5175 \\
					128  & {\bf 24.28} & 35.54 & 96.04
					& {\bf 683} & 1169 & 5175 \\
					256  & {\bf 9.17} & 14.74 & 47.61
					& {\bf 660} & 1169 & 5175 \\
					\bottomrule
				\end{tabular*}
			\end{minipage}
		\end{center}
	\end{table}

	\begin{table}[h]
		\begin{center}
			\begin{minipage}{\textwidth}
	\caption{Average function values and sparsity levels of DSSAL1, ManPG-Ada, and SOC for different tests.}
				\label{tb:spca_aver}
				\begin{tabular*}{\textwidth}{@{\extracolsep{\fill}}ccccccc@{\extracolsep{\fill}}}
					\toprule%
					& \multicolumn{3}{@{}c@{}}{Function value} 
					& \multicolumn{3}{@{}c@{}}{Sparsity level} \\
					\cmidrule{2-4}\cmidrule{5-7}%
					Test & DSSAL1 & ManPG-Ada & SOC 
					& DSSAL1 & ManPG-Ada & SOC \\
					\midrule
					Varying $n$  & -672.26 & -672.26 & -672.26
					& 16.54\%  & 16.51\% &  16.50\% \\
					Varying $p$  & -360.52 & -360.51 & -360.46
					& 40.42\% & 40.46\% & 40.32\% \\
					Varying $\mu$  & -282.65 & -282.65 & -282.64
					& 26.28\% & 26.29\% & 26.28\% \\
					Varying $d$  & -302.02 & -301.97 &  -301.97
					& 22.13\% & 22.21\% & 22.21\% \\
					\bottomrule
				\end{tabular*}
			\end{minipage}
		\end{center}
	\end{table}
		
	It should be evident from Tables \ref{tb:spca_n} to \ref{tb:spca_d} that, 
	in all four test groups and in terms of both wall-clock time and round of communication,
	DSSAL1 clearly outperforms ManPG-Ada which in turn outperforms SOC by large margins.
	Since the amount of communication per round for the three algorithms are essentially the same,
	their total communication overhead is proportional to the rounds of communication required.
	Because DSSAL1 takes far fewer rounds of communication than the other two algorithms 
	(often by considerable margins), we conclude that DSSAL1 is a more 
	communication-efficient algorithm than the other two.  For example, in Table~\ref{tb:spca_mu} 
	for the case of $\mu=0.8$, the number of communication rounds taken by DSSAL1 is less
	than a quarter of that by ManPG-Ada and one tenth of that by COS.
	
\subsection{Empirical convergence rate}

	In this subsection, we examine empirical convergence rates of iterates produced by
	DSSAL1 and ManPG-Ada for comparison, while SOC is excluded from this experiment 
	given its obvious non-competitiveness in previous experiments.
	
	In the following experiments, 
	we fix $n = 1000$, $m = 128000$, $p = 5$, $\mu = 0.2$, and $d = 128$.
	Three synthetic matrices $A \in \Rnm$ is randomly generated by \eqref{eq:gen-A}
	with $\xi $ taking three different values $1.15$, $1.1$, and $1.05$, respectively,
	on which DSSAL1 and ManPG-Ada return $Z^{\ast}_{\mathrm{D}}$ 
	and $Z^{\ast}_{\mathrm{M}}$, respectively, 
	with smaller-than-usual termination tolerances 
	$\epsilon_c = 10^{-8}$ and $\epsilon_g = 10^{-10}np$ in \eqref{eq:stop}.
	We use the average of the two, 
	$Z^{\ast} = (Z^{\ast}_{\mathrm{D}} + Z^{\ast}_{\mathrm{M}}) / 2$,
	as a ``ground truth" solution.
	Then we rerun the two algorithms on the same $A$ 
	with the termination condition $\|Z^{(k)} - Z^{\ast}\|\ff \leq 3 \times 10^{-4}$
	and record the quantity $\|Z^{(k)} - Z^{\ast}\|\ff$ at each iteration.
	
	In Figure \ref{fig:rate}, we plot the iterate-error sequences $\{\|Z^{(k)} - Z^{\ast}\|\ff\}$ 
	for both DSSAL1 and ManPG-Ada and observe that both algorithms appear to converge 
	asymptotically at linear rates.  Overall, however, the convergence of DSSAL1 is 
	several times faster than that of ManPG-Ada.  
	We also provide the ratio between the iteration number of ManPG-Ada and 
	that of DSSAL1 for different values of $\xi$ in Table \ref{tb:iter_ratio}.
	In general, the closer to one $\xi$ is, the slower the singular values of $A$ decay, 
	and the more difficult the problem tends to be.
	Table \ref{tb:iter_ratio} demonstrates that in our test the advantage of DSSAL1 becomes 
	more and more pronounced as the test instances become more and more difficult to solve.

	\begin{figure}[h]
		\centering
		\subfloat[$\xi = 1.15$]{
			\label{subfig:xi_15}
			\includegraphics[width=0.3\linewidth]{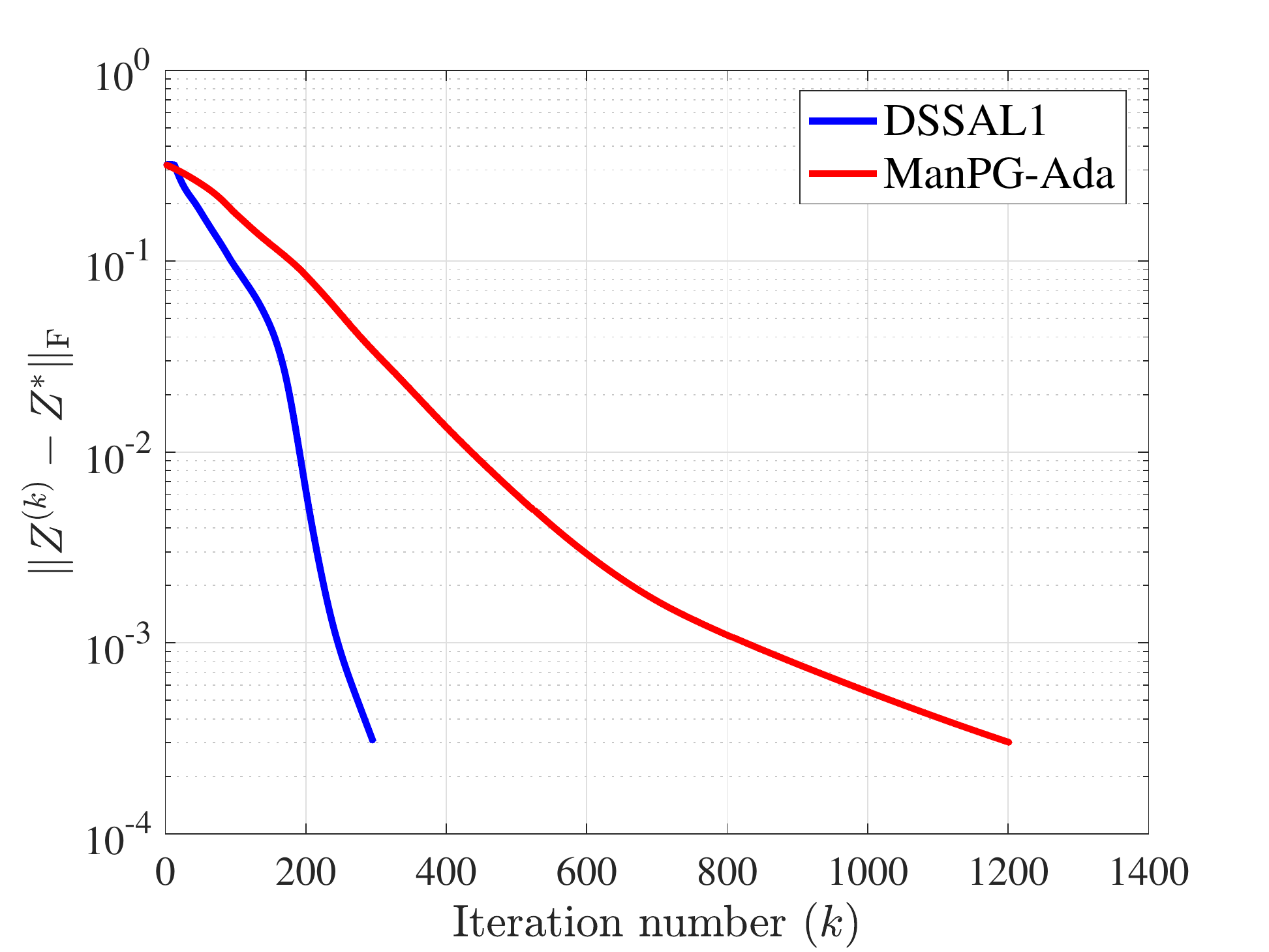}
		}
		\subfloat[$\xi = 1.1$]{
			\label{subfig:xi_10}
			\includegraphics[width=0.3\linewidth]{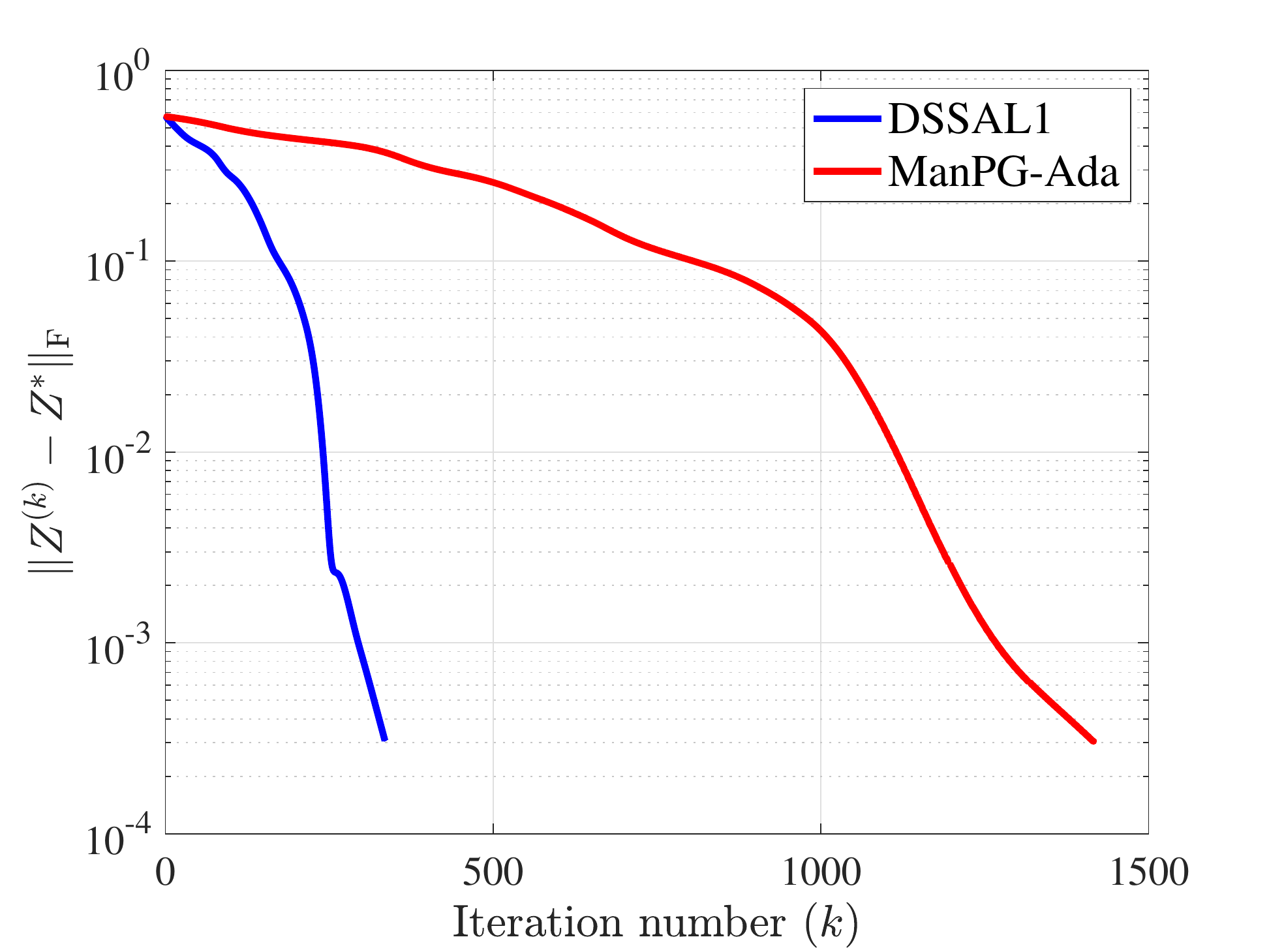}
		}
		\subfloat[$\xi = 1.05$]{
			\label{subfig:xi_05}
			\includegraphics[width=0.3\linewidth]{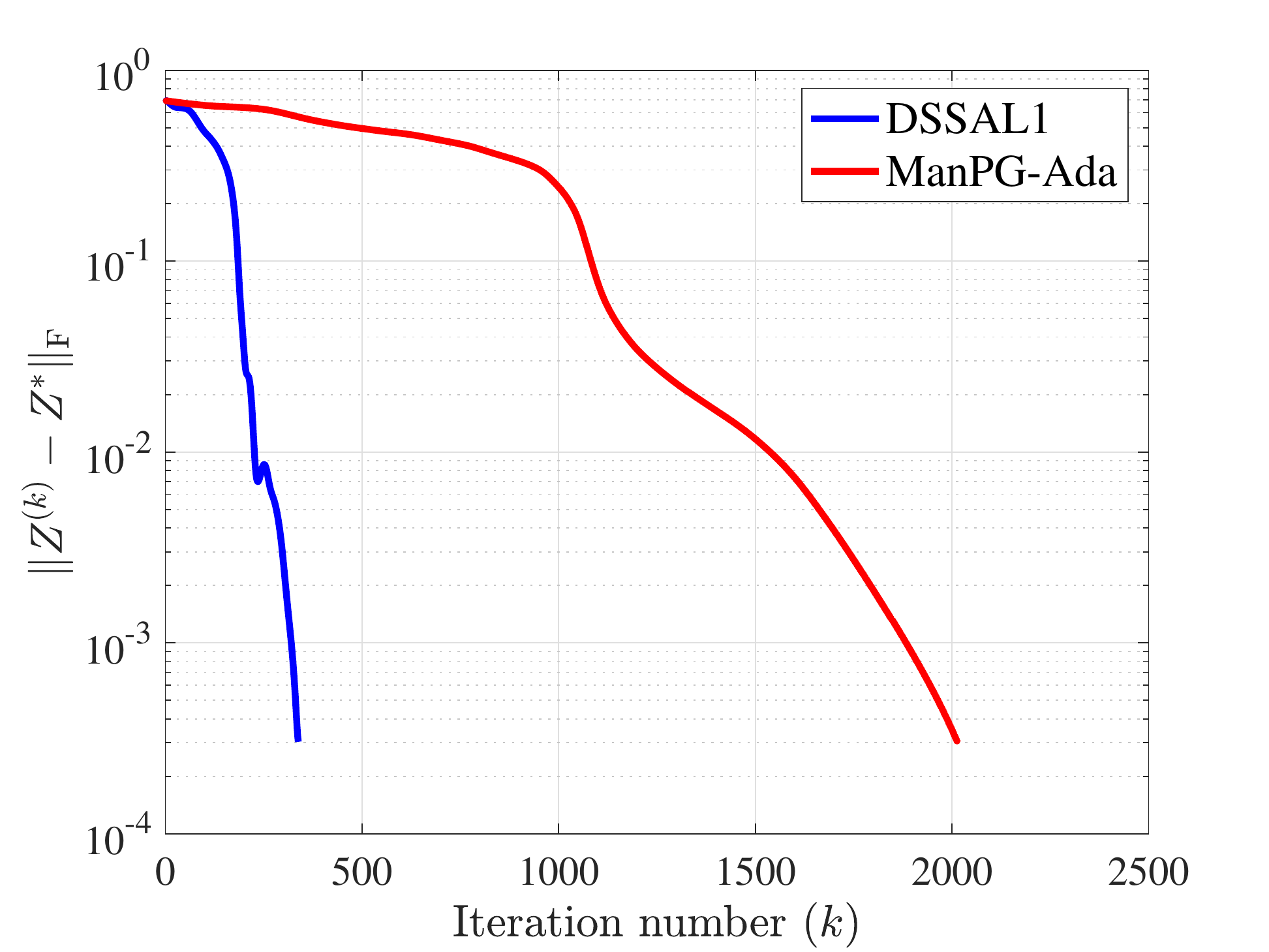}
		}
		
		\caption{Comparison between DSSAL1 and ManPG-Ada of empirical convergence rates.}
		\label{fig:rate}
	\end{figure}
	
	\begin{table}[ht]
		\begin{center}
			\begin{minipage}{\textwidth}
				\caption{The iteration number ratio between ManPG-Ada and DSSAL1 for different values of $\xi$.}
				\label{tb:iter_ratio}
				\begin{tabular*}{\textwidth}{@{\extracolsep{\fill}}cccc@{\extracolsep{\fill}}}
					\toprule%
					& $\xi = 1.15$ & $\xi = 1.1$ & $\xi = 1.05$ \\
					\midrule
					$\dfrac{\mathrm{It}_\mathrm{ManPG-Ada}}{\mathrm{It}_\mathrm{DSSAL1}}$ 
					& $\dfrac{1202}{296} \approx 4.06$ 
					& $\dfrac{1417}{335} \approx 4.23$ 
					& $\dfrac{2014}{338} \approx 5.96$ \\
					\bottomrule
				\end{tabular*}
			\end{minipage}
		\end{center}
	\end{table}

\section{Conclusion}

\label{sec:conclusion}

	In this paper, we propose a distributed algorithm, called DSSAL1, for solving the $\ell_1$-regularized 
	optimization model \eqref{eq:opt-spca-l1} for sparse PCA computation. 
	DSSAL1 has the following features.
	\begin{enumerate}
	\item The algorithm successfully extends the subspace-splitting strategy from orthogonal-transformation invariant
	objective function $f=\sum f_i$ to the sum $f + r$ where $r$ can be non-invariant and non-smooth.
	\item The algorithm has built-in mechanism for local-data masking and hence naturally protects local-data privacy 
	without requiring a special privacy preservation process.   
	\item The algorithm is storage-efficient in that beside local data, it only requires storing  
	$n\times p$ matrices (usually $p \ll n$) by each agent, thanks to a low-rank multiplier formula.
	\item The algorithm has a global convergence guarantee to stationary points and a 
	worst-case complexity under mild conditions, in spite of the nonlinear equality constraints
	for subspace consensus.

	\end{enumerate}

	Comprehensive numerical simulations are conducted under a distributed environment 
	to evaluate the performance of our algorithm in comparison to two state-of-the-art approaches.  
	Remarkably, the communication rounds required by DSSAL1 
	are often over one order of magnitude smaller than existing methods.
	These results indicate that DSSAL1 has a great potential in solving large-scale application
	problems in distributed environments where data privacy is a primary concern.
	
	Finally, we mention two related topics worthy of future studies.  One is the possibility of developing 
	asynchronous approaches for sparse PCA to address load balance issues in distributed environments.
	Another is to extend the subspace splitting strategy to decentralized networks so that a wider range of
	applications can benefit from the effectiveness of this approach.

\appendix

\appendixpage

\begin{appendices}

\section{Proof of Lemma \ref{le:kkt}}

\label{apx:first-order}

	\begin{proof}[Proof of Lemma \ref{le:kkt}]
		According to the definition of $\projts{Z}{\cdot}$, it follows that
		\begin{equation*}
			\begin{aligned}
				& \norm{ \projts{Z}{ -A A\zz Z + R(Z)} }\fs \\
				= {} & \dfrac{1}{4}\norm{Z\zz \dkh{-A A\zz Z + R(Z)} - \dkh{-A A\zz Z + R(Z)}\zz Z}\fs \\
				& + \norm{ \Pv{Z} \dkh{-A A\zz Z + R(Z)} }\fs \\
				= {} & \norm{ \Pv{Z} \dkh{-A A\zz Z + R(Z)} }\fs 
				+ \dfrac{1}{4}\norm{Z\zz R(Z) - R(Z)\zz Z}\fs,
			\end{aligned}		
		\end{equation*}
		where $R(Z) \in \partial r(Z)$.
		The proof is completed.
	\end{proof}

\section{Proof of Proposition \ref{prop:kkt-multipliers}}

\label{apx:kkt-ps}

	\begin{proof}[Proof of Proposition \ref{prop:kkt-multipliers}]
		To begin with, we assume that $\dkh{ Z, \{X_i\} }$ is a first-order stationary point.
		Then there exists $R(Z) \in \partial r(Z)$ such that
		\begin{equation*}
			\Pv{Z} \dkh{- A A\zz Z + R(Z)} = 0,
		\end{equation*}
		and $Z\zz R(Z)$ is symmetric.
		Let $\Theta = Z\zz R(Z) \in Z\zz \partial r(Z)$, $\Gamma_i = - X_i\zz A_i A_i\zz X_i$, and 
		\begin{equation*}
			\Lambda_i = - \Pv{X_i} A_i A_i\zz X_i X_i\zz - X_i X_i\zz A_i A_i\zz \Pv{X_i}
		\end{equation*}
		with $\iid$.  
		Then the matrices $\Theta$, $\Gamma_i$ and $\Lambda_i$ are symmetric and $\rank \dkh{\Lambda_i} \leq 2p$. 
		Moreover, we can deduce that 
		\begin{equation*}
			A_i A_i\zz X_i + X_i \Gamma_i + \Lambda_i X_i 
			=A_i A_i\zz X_i - X_iX_i\zz A_i A_i\zz X_i - \Pv{X_i} A_i A_i\zz X_i = 0,
		\end{equation*}
		and
		\begin{equation*}
			\begin{aligned}
				R(Z) + \sumiid\Lambda_i Z - Z \Theta
				= {} & R(Z) - \sumiid \Pv{Z} A_i A_i\zz Z - ZZ\zz R(Z) \\
				= {} & \Pv{Z} \dkh{- A A\zz Z + R(Z)} = 0.
			\end{aligned}
		\end{equation*}
		Hence, $\dkh{ Z, \{X_i\} }$ satisfies the conditions in \eqref{eq:kkt-multipliers} 
		under these specific choices of $\Theta$, $\Gamma_i$ and $\Lambda_i$.
		
		Conversely, we now assume that there exist $R(Z) \in \partial r(Z)$ 
		and symmetric matrices $\Theta$, $\Gamma_i$ and $\Lambda_i$ 
		such that $\dkh{ Z, \{X_i\} }$ satisfies the conditions in \eqref{eq:kkt-multipliers}. 
		It follows from the first and second equality in \eqref{eq:kkt-multipliers} that
		\begin{equation*}
			\begin{aligned}
				& \sumiid \Pv{X_i} A_i A_i\zz X_i X_i\zz 
				= - \sumiid \Pv{X_i} \dkh{ X_i\Gamma_i + \Lambda_i X_i } X_i\zz
				=  - \sumiid \Pv{X_i} \Lambda_i X_i X_i\zz \\
				& = - \Pv{Z} \dkh{ \sumiid\Lambda_i Z } Z\zz
				= \Pv{Z}\dkh{R(Z) - Z\Theta} Z\zz
				= \Pv{Z} R(Z) Z\zz.
			\end{aligned}
		\end{equation*}
		At the same time, since $X_i X_i\zz = Z Z\zz$, we have 
		\begin{align*}
			\sumiid \Pv{X_i} A_i A_i\zz X_i X_i\zz 
			={} & \sumiid \Pv{Z} A_i A_i\zz Z Z\zz = \Pv{Z} A A\zz Z Z\zz.
		\end{align*}
		Combining the above two equalities and orthogonality of $Z$, we arrive at
		\begin{equation*}
			\Pv{Z} \dkh{- A A\zz Z + R(Z)} = 0.
		\end{equation*}
		Left-multiplying both sides of the second equality in \eqref{eq:kkt-multipliers} by $Z\zz$,
		we obtain that 
		\begin{equation*}
			Z\zz R(Z) = \Theta - \sumiid Z\zz \Lambda_i Z,
		\end{equation*}
		which together with the symmetry of $\Lambda_i$ and $\Theta$ 
		implies that $Z\zz R(Z)$ is also symmetric.
		This completes the proof.
	\end{proof}

\section{Proof of Lemma \ref{le:optimality}}

\label{apx:optimality}

	\begin{proof}[Proof of Lemma \ref{le:optimality}]
		Since $( Z^{(k)}, \{X_i^{(k)}\} )$ is feasible, 
		we know $X_i^{(k)} (X_i^{(k)})\zz = Z^{(k)} (Z^{(k)})\zz$ for $\iid$.
		Thus, it can be readily verified that
		\begin{equation*}
			\begin{aligned}
				Q^{(k)} Z^{(k)} 
				= {} & \sumiid \dkh{\Lambda_i^{(k)} - \beta_i X_i^{(k)} (X_i^{(k)})\zz} Z^{(k)} \\
				= {} & \sumiid \dkh{- \Pv{Z^{(k)}} A_i A_i\zz Z^{(k)} (Z^{(k)})\zz - \beta_i Z^{(k)} (Z^{(k)})\zz} Z^{(k)} \\
				= {} & - \Pv{Z^{(k)}} A_i A_i\zz Z^{(k)} - \dkh{ \sumiid \beta_i } Z^{(k)},
			\end{aligned}
		\end{equation*}
		which implies that
		\begin{equation*}
			\projts{Z^{(k)}}{Q^{(k)}Z^{(k)}} = \projts{Z^{(k)}}{- A_i A_i\zz Z^{(k)}}.
		\end{equation*}
		According to Theorem 4.1 in \cite{Yang2014}, the first-order optimality condition 
		of  \eqref{eq:ps-sub-z-pg} can be stated as:
		\begin{equation*}
			0 \in \projts{Z^{(k)}}{Q^{(k)} Z^{(k)} + \dfrac{1}{\eta} D^{(k)} + \partial r(Z^{(k)} + D^{(k)})}.
		\end{equation*}
		Since $D^{(k)} = 0$ is the global minimizer of  \eqref{eq:ps-sub-z-pg}, we have
		\begin{equation*}
			\begin{aligned}
				0 \in {} & \projts{Z^{(k)}}{Q^{(k)} Z^{(k)} + \partial r(Z^{(k)})} \\
				= {} & \projts{Z^{(k)}}{- A_i A_i\zz Z^{(k)} + \partial r(Z^{(k)})}.
			\end{aligned}
		\end{equation*}
		We obtain the assertion of this lemma.
	\end{proof}

\section{Convergence of Algorithm \ref{alg:DSSAL1}}

\label{apx:global}

	Now we prove Theorem \ref{thm:global} 
	to establish the global convergence of Algorithm \ref{alg:DSSAL1}.
	In addition to the notations introduced in Section \ref{sec:introduction}, 
	we further adopt the followings throughout the theoretical analysis.
	The notations \(\rank \dkh{C}\) and \(\sigma_{\min} \dkh{C}\)  
	represent the rank and the smallest singular value of \(C\), respectively.
	For \(X, Y \in \stiefel\), we define \(\dsp{X}{Y} := XX\zz - YY\zz\)
	and \(\distp{X}{Y} := \norm{ \dsp{X}{Y} }\ff\), standing for, respectively, 
	the projection distance matrix and its measurement.
	
	To begin with, we provide a sketch of our proof. 
	Suppose $\{Z^{(k)}\}$ is the iteration sequence generated by Algorithm \ref{alg:DSSAL1},
	with $X_i^{(k)}$ and $\Lambda_i^{(k)}$ being the local variable and multiplier 
	of the $i$-th agent at the $k$-th iteration, respectively.  
	The proof includes the following main steps.

	\begin{enumerate}
		
		\item The sequence $\{Z^{(k)}\}$ is bounded 
		and the sequence $\{ \cL( Z^{(k)}, \{X_i^{(k)}\}, \{\Lambda_i^{(k)}\} ) \}$ is bounded from below.
		
		\item The sequence $\{Z^{(k)}\}$ satisfies 
		$\dists{Z^{(k+1)}}{X_i^{(k)}} \leq 2 ( 1 -\underline{\sigma}^2 )$, 
		and $\underline{\sigma}$  is a unified lower bound 
		of the smallest singular values of the matrices $(X_i^k)\zz Z^{k+1}(\iid)$.
		
		\item The sequence $\{ \cL ( Z^{(k)}, \{X_i^{(k)}\}, \{\Lambda_i^{(k)}\} ) \}$ 
		is monotonically non-increasing, 
		and hence is convergent.
		
		\item The sequence $\{Z^{(k)}\}$ has at least one accumulation point, 
		and any accumulation point is a first-order stationary point of 
		the sparse PCA problem \eqref{eq:opt-spca-l1}.
		
	\end{enumerate}
	
	Next we verify all the items in the above sketch by proving the following lemmas and corollaries.

	\begin{lemma}
		\label{le:dist-gk}
		Suppose $\{Z^{(k)}\}$ is the iterate sequence generated by Algorithm \ref{alg:DSSAL1}.
		Let 
		\begin{equation*}
			g^{(k)} (D) = \jkh{Q^{(k)}Z^{(k)}, D} + \dfrac{1}{2\eta} \norm{D}\fs + r(Z^{(k)} + D).
		\end{equation*}
		Then the following relationship holds for any $k \in \N$,
		\begin{equation*}
			g^{(k)} (0) - g^{(k)} (D^{(k)}) \geq \dfrac{1}{2\eta} \norm{D^{(k)}}\fs.
		\end{equation*}
	\end{lemma}

	\begin{proof}
		Since $g^{(k)}$ is strongly convex with modulus $\dfrac{1}{\eta}$, we have
		\begin{equation}
			\label{eq:gk-sc}
			g^{(k)} (\hat{D}) 
			\geq g^{(k)} (D) + \jkh{\partial g^{(k)} (D), \hat{D} - D} 
			+ \dfrac{1}{2\eta} \norm{\hat{D} - D}\fs,
		\end{equation}
		for any $D, \hat{D} \in \Rnp$.
		In particular, if $\hat{D}, D \in \tanst{Z^{(k)}}$, it holds that
		\begin{equation*}
			\jkh{\partial g^{(k)} (D), \hat{D} - D} 
			= \jkh{\projts{Z^{(k)}}{\partial g^{(k)} (D)}, \hat{D} - D}.
		\end{equation*}
		It follows from the first-order optimality condition of \eqref{eq:ps-sub-z-pg} 
		that $0 \in \projts{Z^{(k)}}{\partial g^{(k)} (D^{(k)})}$.
		Finally, taking $\hat{D} = 0$ and $D = D^{(k)}$ in \eqref{eq:gk-sc} 
		yields the assertion of this lemma.
	\end{proof}

	\begin{lemma}
		\label{le:polar}
		Suppose $Z \in \stiefel$ and $D \in \tanst{Z}$.
		Then it holds that
		\begin{equation*}
			\norm{\proj_{\stiefel}\dkh{Z + D} - Z}\ff \leq \norm{D}\ff,
		\end{equation*}
		and
		\begin{equation*}
			\norm{\proj_{\stiefel}\dkh{Z + D} - Z - D}\ff \leq \dfrac{1}{2} \norm{D}\fs.
		\end{equation*}
	\end{lemma}

	\begin{proof}
		The proof can be found in, for example, \cite{Jiang2017}.
		For the sake of completeness, we provide a proof here.
		It follows from the orthogonality of $Z$ and the skew-symmetry of $Z\zz D$ 
		that $Z + D$ has full column rank.
		This yields that $\proj_{\stiefel} (Z + D) = (Z + D)F\inv$,
		where $F = (I_p + D\zz D)^{1/2}$.
		Since $\proj_{\stiefel} (Z + D) - Z = ( Z (I_p - F) + D ) F\inv$,
		we have
		\begin{equation*}
			\begin{aligned}
				\norm{\proj_{\stiefel}\dkh{Z + D} - Z}\fs 
				= {} & 2\tr\dkh{I_p - F\inv} - 2 \tr \dkh{F\inv Z\zz D}
				= 2\tr\dkh{I_p - F\inv} \\
				= {} & 2 \sumjjd \dkh{ 1 - \dkh{ 1 + \tilde{\sigma}_i^2 }^{-1/2} }
				\leq \sumjjd \tilde{\sigma}_i^2 = \norm{D}\fs,
			\end{aligned}
		\end{equation*}
		where $\tilde{\sigma}_1 \geq \dotsb \geq \tilde{\sigma}_d \geq 0$ are the singular values of $D$.
		Similarly, it follows from the relationship 
		$\proj_{\stiefel} (Z + D) - Z - D = (Z + D) (F\inv - I_p)$ that
		\begin{equation*}
			\begin{aligned}
				\norm{\proj_{\stiefel}\dkh{Z + D} - Z - D}\ff 
				= {} & \tr \dkh{ \dkh{I_p - F}^2 } 
				= \sumjjd \dkh{ 1 - \dkh{1 + \tilde{\sigma}_i^2}^{1/2} }^2 \\
				\leq {} & \dfrac{1}{4} \sumjjd \tilde{\sigma}_i^4 = \dfrac{1}{4} \norm{D}\ff^4,
			\end{aligned}
		\end{equation*}
		which completes the proof.
	\end{proof}

	\begin{corollary}
		\label{cor:des-z}
		Suppose $\{Z^{(k)}\}$ is the iterate sequence generated by Algorithm \ref{alg:DSSAL1}
		with the parameters satisfying Condition \ref{asp:step}.
		Then for any $k \in \N$, it holds that
		\begin{equation}
			\label{eq:dist-lag-z}
			\cL( Z^{(k)}, \{X_i^{(k)}\}, \{\Lambda_i^{(k)}\} ) - \cL( Z^{(k+1)}, \{X_i^{(k)}\}, \{\Lambda_i^{(k)}\} ) 
			\geq \bar{M} \norm{D^{(k)}}\fs,
		\end{equation}
		where $\bar{M} > 0$ is a constant defined in Section \ref{sec:convergence-analysis}. 
	\end{corollary}

	\begin{proof}
		Firstly, it can be readily verified that
		\begin{equation*}
			\norm{Q^{(k)}}\ff \leq \sumiid \norm{Q_i^{(k)}}\ff 
			\leq \sumiid \dkh{ 2 \norm{A_i}\fs + \sqrt{p} \beta_i }.
		\end{equation*}
		Let $\bar{q}^{(k)} (Z) = \tr (Z\zz Q^{(k)} Z) / 2$ be the smooth part 
		of the objective function $q^{(k)} (Z)$ in  \eqref{eq:ps-sub-z}.
		Since $\nabla \bar{q}^{(k)}$ is Lipschitz continuous 
		with the corresponding Lipschitz constant $\norm{Q^{(k)}}\ff $,
		we have
		\begin{equation*}
			\begin{aligned}
				\bar{q}^{(k)} (Z^{(k+1)}) - \bar{q}^{(k)} (Z^{(k)}) 
				\leq {} & \jkh{Q^{(k)} Z^{(k)}, Z^{(k+1)} - Z^{(k)}} \\
				& + \dfrac{1}{2} \norm{Q^{(k)}}\ff \norm{Z^{(k+1)} - Z^{(k)}}\fs.
			\end{aligned}
		\end{equation*}
		It follows from Lemma \ref{le:polar} that
		\begin{equation*}
			\begin{aligned}
				\jkh{Q^{(k)} Z^{(k)}, Z^{(k+1)} - Z^{(k)} - D^{(k)}}
				\leq {} & \norm{Q^{(k)} Z^{(k)}}\ff \norm{Z^{(k+1)} - Z^{(k)} - D^{(k)}}\ff \\
				\leq {} & \sumiid \dkh{ \norm{A_i}\fs + \dfrac{\sqrt{p}}{2} \beta_i } \norm{D^{(k)}}\fs,
			\end{aligned}
		\end{equation*}
		and
		\begin{equation*}
			\begin{aligned}
				\dfrac{1}{2} \norm{Q^{(k)}}\ff \norm{Z^{(k+1)} - Z^{(k)}}\fs
				\leq \sumiid \dkh{ \norm{A_i}\fs + \dfrac{\sqrt{p}}{2} \beta_i } \norm{D^k}\fs.
			\end{aligned}
		\end{equation*}
		Combing the above three inequalities, we can obtain that
		\begin{equation*}
			\begin{aligned}
				\bar{q}^{(k)} (Z^{(k+1)}) - \bar{q}^{(k)} (Z^{(k)}) 
				\leq \jkh{Q^{(k)} Z^{(k)}, D^{(k)}} 
				+ \sumiid \dkh{ 2\norm{A_i}\fs + \beta_i \sqrt{p} } \norm{D^{(k)}}\fs.
			\end{aligned}
		\end{equation*}
		It follows from Lemma \ref{le:dist-gk} that
		\begin{equation*}
			\begin{aligned}
				& \jkh{Q^{(k)} Z^{(k)}, D^{(k)}} + r(Z^{(k)} + D^{(k)}) - r(Z^{(k)}) \\
				& = g^{(k)} (D^{(k)}) - g^{(k)} (0) - \dfrac{1}{2\eta} \norm{D^{(k)}}\fs
				\leq - \dfrac{1}{\eta} \norm{D^{(k)}}\fs,
			\end{aligned}
		\end{equation*}
		which infers that
		\begin{equation*}
			\begin{aligned}
				& \bar{q}^{(k)} (Z^{(k+1)}) - \bar{q}^{(k)} (Z^{(k)}) + r(Z^{(k)} + D^{(k)}) - r(Z^{(k)}) \\
				& \leq \sumiid \dkh{ 2\norm{A_i}\fs 
				+ \beta_i \sqrt{p} } \norm{D^{(k)}}\fs - \dfrac{1}{\eta} \norm{D^{(k)}}\fs.
			\end{aligned}
		\end{equation*}
		This together with the Lipschitz continuity of $r(Z)$ yields that
		\begin{equation*}
			\begin{aligned}
				& q^{(k)} (Z^{(k+1)}) - q^{(k)} (Z^{(k)}) \\
				& = \bar{q}^{(k)} (Z^{(k+1)}) - \bar{q}^{(k)} (Z^{(k)}) + r(Z^{(k+1)}) - r(Z^{(k)}) \\
				& = \bar{q}^{(k)} (Z^{(k+1)}) - \bar{q}^{(k)} (Z^{(k)}) 
				+ r(Z^{(k)} + D^{(k)}) - r(Z^{(k)}) \\
				& \quad + r(Z^{(k+1)}) - r(Z^{(k)} + D^{(k)}) \\
				& \leq \bar{q}^{(k)} (Z^{(k+1)}) - \bar{q}^{(k)} (Z^{(k)})  
				+ r(Z^{(k)} + D^{(k)}) - r(Z^{(k)}) \\
				& \quad + \mu \sqrt{np} \norm{Z^{(k+1)} - Z^{(k)} - D^{(k)}}\ff \\
				& \leq \sumiid \dkh{ 2\norm{A_i}\fs + \beta_i \sqrt{p} } \norm{D^{(k)}}\fs 
				- \dfrac{1}{\eta} \norm{D^{(k)}}\fs + \dfrac{\mu}{2} \sqrt{np} \norm{D^{(k)}}\fs  \\
				& = \dkh{ \bar{M} - \dfrac{1}{\eta} } \norm{D^{(k)}}\fs.
			\end{aligned}
		\end{equation*}
		Here, $\bar{M} > 0$ is a constant defined in Section \ref{sec:convergence-analysis}.
		According to Condition \ref{asp:step}, we know that $\bar{M} - 1/\eta \leq -\bar{M}$.
		Hence, we finally arrive at
		\begin{equation*}
			\begin{aligned}
				\cL( Z^{(k)}, \{X_i^{(k)}\}, \{\Lambda_i^{(k)}\}) - \cL( Z^{(k+1)}, \{X_i^{(k)}\}, \{\Lambda_i^{(k)}\}) 
				= {} & q^{(k)} (Z^{(k)}) - q^{(k)} (Z^{(k+1)}) \\
				\geq {} & \bar{M} \norm{D^{(k)}}\fs.
			\end{aligned}
		\end{equation*}
		This completes the proof.
	\end{proof}

	\begin{lemma}
		\label{le:dist-zk}
		Suppose $\{Z^{(k)}\}$ is the iterate sequence generated by Algorithm \ref{alg:DSSAL1}
		with the parameters satisfying Condition \ref{asp:step}.
		Then for any $k \in \N$, it can be verified that
		\begin{equation}\label{eq:dist-zk}
			\dists{Z^{(k+1)}}{X_i^{(k)}}
			\leq \rho \sumjjd \dists{Z^{(k)}}{X_j^{(k)}}
			+ \dfrac{8}{\beta_i} \dkh{\sqrt{p} \norm{A}\fs + \mu n p},
		\end{equation}
		where $\rho \geq 1$ is a constant defined in Section \ref{sec:convergence-analysis}.
	\end{lemma}

	\begin{proof}
		The inequality \eqref{eq:dist-lag-z} directly results in 
		the following relationship. 
		\begin{equation*}
			{q}^{(k)} (Z^{(k)}) - {q}^{(k)} (Z^{(k+1)}) \geq 0.
		\end{equation*}
		According to the definition of $q^{(k)}$, it follows that
		\begin{equation*}
			\begin{aligned}
				0 \leq {} & \dfrac{1}{2} \tr \dkh{ (Z^{(k)})\zz Q^{(k)} Z^{(k)} } 
				- \dfrac{1}{2} \tr \dkh{ (Z^{(k+1)})\zz Q^{(k)} Z^{(k+1)} } 
				+ r(Z^{(k)}) - r(Z^{(k+1)}) \\
				\leq {} & \dfrac{1}{2} \sumjjd \tr \dkh{ \dkh{ \beta_j X_j^{(k)} (X_j^{(k)})\zz 
				- \Lambda_j^{(k)} }\dsp{Z^{(k+1)}}{Z^{(k)}} }
				+ 2 \mu n p.
			\end{aligned}
		\end{equation*}
		By straightforward calculations, we can deduce that
		\begin{equation*}
			\begin{aligned}
				\sumjjd \tr \dkh{  \Lambda_j^{(k)} \dsp{Z^{(k)}}{Z^{(k+1)}} } 
				\leq {} & \sumjjd \norm{\Lambda_j^{(k)}}\ff \distp{Z^{(k+1)}}{Z^{(k)}} \\
				\leq {} & 4 \sqrt{p} \sumjjd \norm{A_j}\fs
				= 4\sqrt{p} \norm{A}\fs,
			\end{aligned}
		\end{equation*}
		and
		\begin{equation*}
			\begin{aligned}
				& \sumjjd \beta_j \tr \dkh{ X_j^{(k)} (X_j^{(k)})\zz \dsp{Z^{(k+1)}}{Z^{(k)}} } \\
				& = \dfrac{1}{2} \sumjjd \beta_j \dists{Z^{(k)}}{X_j^{(k)}}
				- \dfrac{1}{2} \sumjjd \beta_j \dists{Z^{(k+1)}}{X_j^{(k)}}.
			\end{aligned}
		\end{equation*}
		The above three inequalities yield that
		\begin{equation*}
			\sumjjd \beta_j \dists{Z^{(k+1)}}{X_j^{(k)}}
			\leq \sumjjd \beta_j \dists{Z^{(k)}}{X_j^{(k)}} + 8 \sqrt{p} \norm{A}\fs + 8 \mu n p,
		\end{equation*}
		which further implies that
		\begin{equation*}
			\begin{aligned}
				\dists{Z^{(k+1)}}{X_i^{(k)}}
				\leq {} & \dfrac{1}{\beta_i}\sumjjd \beta_j \dists{Z^{(k+1)}}{X_j^{(k)}} \\
				\leq {} & \rho \sumjjd \dists{Z^{(k)}}{X_j^{(k)}} 
				+ \dfrac{8}{\beta_i}\dkh{\sqrt{p} \norm{A}\fs + \mu n p}.
			\end{aligned}
		\end{equation*}
		This completes the proof.
	\end{proof}

	\begin{lemma}\label{le:des-x}
		Suppose $Z^{(k+1)}$ is the $(k+1)$-th iterate generated by Algorithm \ref{alg:DSSAL1} 
		and satisfies the following condition: 
		\begin{equation*}
			\dists{Z^{(k+1)}}{X_i^{(k)}} \leq 2 \dkh{1 -\underline{\sigma}^2},
		\end{equation*}
		where $\underline{\sigma} \in (0,1)$ is a constant defined in Condition \ref{asp:step}. 
		Let the algorithm parameters satisfy Conditions \ref{asp:step} and \ref{asp:beta}. 
		Then for any $\iid$, it holds that
		\begin{equation}\label{eq:des-x}
			h_i^{(k)} ( X_i^{(k)} ) - h_i^{(k)} ( X_i^{(k+1)} ) 
			\geq \frac{1}{4} \underline{\sigma}^2 c_i \beta_i \dists{Z^{(k+1)}}{X_i^{(k)}},
		\end{equation}
		and
		\begin{equation}\label{eq:dist-xk+1-zk}
			\dists{Z^{(k+1)}}{X_i^{(k+1)}}
			\leq \dkh{1 - c_i \underline{\sigma}^2} \dists{Z^{(k+1)}}{X_i^{(k)}} 
			+ \dfrac{12}{\beta_i} \sqrt{p} \norm{A_i}\fs.
		\end{equation}
	\end{lemma}

	\begin{proof}
		It follows from Condition \ref{asp:beta} that $\beta_i > c_i^{\prime} \norm{A_i}_2^2$, 
		which together with \eqref{eq:ps-sub-x-con-1} yields that
		\begin{equation}\label{eq:des-x-1}
			h_i^{(k)} ( X_i^{(k)} ) - h_i^{(k)} ( X_i^{(k+1)} ) 
			\geq \dfrac{c_i}{2\beta_i} \norm{ \Pv{X_i^{(k)}} H_i^{(k)} X_i^{(k)} }\fs,
		\end{equation}
		And it can be checked that
		\begin{equation}
			\label{eq:rgrad-h}
			\begin{aligned}
				& \Pv{X_i^{(k)}} H_i^{(k)} X_i^{(k)}
				= \Pv{X_i^{(k)}} \dkh{ A_i A_i\zz X_i^{(k)} 
				+ \Lambda_i^{(k)} X_i^{(k)} + \beta_i Z^{(k+1)}(Z^{(k+1)})\zz X_i^{(k)} } \\
				& = \Pv{X_i^{(k)}} \dkh{ A_i  A_i\zz X_i^{(k)} - \Pv{X_i^{(k)}} A_i A_i\zz X_i^{(k)} }
				-\beta_i \Pv{X_i^{(k)}} Z^{(k+1)}(Z^{(k+1)})\zz X_i^{(k)} \\
				& = -\beta_i \Pv{X_i^{(k)}} Z^{(k+1)}(Z^{(k+1)})\zz X_i^{(k)}.
			\end{aligned}
		\end{equation}
		Suppose $\hat{\sigma}_1, \dotsc, \hat{\sigma}_p$ are the singular values of $(X_i^{(k)})\zz Z^{(k+1)}$.
		It is clear that $0 \leq \hat{\sigma}_i \leq 1$ for any $i = 1, \dotsc, p$
		due to the orthogonality of $X_i^{(k)}$ and $Z^{(k+1)}$.
		On the one hand, we have
		\begin{equation*}
			\dists{Z^{(k+1)}}{X_i^{(k)}}
			= \norm{ X_i^{(k)}(X_i^{(k)})\zz - Z^{(k+1)}(Z^{(k+1)})\zz }\fs 
			= 2\sum\limits_{j=1}^p \dkh{ 1 - \hat{\sigma}_j^2 }.
		\end{equation*}
		On the other hand, it follows from $\dists{Z^{(k+1)}}{X_i^{(k)}} \leq 2 \dkh{1 -\underline{\sigma}^2}$ that
		\begin{equation*}
			\sigma_{\min}\dkh{ (X_i^{(k)})\zz Z^{(k+1)} } \geq \underline{\sigma}.
		\end{equation*}
		Let $Y_i^{(k)} = (X_i^{(k)})\zz Z^{(k+1)}(Z^{(k+1)})\zz X_i^{(k)}$. 
		By straightforward calculations, we can derive that
		\begin{equation}\label{eq:pzzx}
			\begin{aligned}
				& \norm{\Pv{X_i^{(k)}} Z^{(k+1)}(Z^{(k+1)})\zz X_i^{(k)}}\fs 
				= \tr\dkh{Y_i^{(k)}} - \tr\dkh{(Y_i^{(k)})^2}
				= \sum\limits_{j=1}^p \hat{\sigma}_j^2 \dkh{ 1 - \hat{\sigma}_j^2 } \\
				& \geq \sum\limits_{j=1}^p \sigma_{\min}^2\dkh{ (X_i^{(k)})\zz Z^{(k+1)} } 
				\dkh{ 1 - \hat{\sigma}_j^2 }
				\geq \dfrac{1}{2} \underline{\sigma}^2 \dists{Z^{(k+1)}}{X_i^{(k)}}.
			\end{aligned}
		\end{equation}
		Combining \eqref{eq:des-x-1}, \eqref{eq:rgrad-h} and \eqref{eq:pzzx}, 
		we acquire the assertion \eqref{eq:des-x}.
		Then it follows from the definition of $h_i^{(k)}$ that
		\begin{equation*}
			\begin{aligned}
				c_i \underline{\sigma}^2 \dists{Z^{(k+1)}}{X_i^{(k)}}
				\leq {} & 2 \tr \dkh{ Z^{(k+1)}(Z^{(k+1)})\zz \dsp{X_i^{(k+1)}}{X_i^{(k)}} } \\
				& + \dfrac{2}{\beta_i} \tr \dkh{ \dkh{ A_i A_i\zz + \Lambda_i^{(k)} } 
				\dsp{X_i^{(k+1)}}{X_i^{(k)}} }.
			\end{aligned}
		\end{equation*}
		By straightforward calculations, we can obtain that
		\begin{equation*}
			\begin{aligned}
				\tr \dkh{ \dkh{ A_i A_i\zz + \Lambda_i^{(k)} } \dsp{X_i^{(k+1)}}{X_i^{(k)}} }
				\leq {} & \norm{ A_i A_i\zz + \Lambda_i^{(k)} }\ff \distp{X_i^{(k+1)}}{X_i^{(k)}} \\
				\leq {} & 6 \sqrt{p}\norm{A_i}\fs,
			\end{aligned}
		\end{equation*}
		and
		\begin{equation*}
			\begin{aligned}
				& \tr \dkh{ Z^{(k+1)}(Z^{(k+1)})\zz \dsp{X_i^{(k+1)}}{X_i^{(k)}} } \\
				& = \dfrac{1}{2} \dists{Z^{(k+1)}}{X_i^{(k)}}
				- \dfrac{1}{2} \dists{Z^{(k+1)}}{X_i^{(k+1)}}.
			\end{aligned}
		\end{equation*}
		The above three relationships yield \eqref{eq:dist-xk+1-zk}.
		We complete the proof.
	\end{proof}

	\begin{lemma}\label{le:svd-xz}
		Suppose $\{Z^{(k)}\}$ is the iterate sequence generated by Algorithm \ref{alg:DSSAL1}
		initiated from $Z^{(0)} \in \stiefel$
		with the parameters satisfying Conditions \ref{asp:step} and \ref{asp:beta}.
		Then for any $\iid$ and $k \in \N$, it holds that 
		\begin{equation}\label{eq:svd-xz}
			\dists{Z^{(k+1)}}{X_i^{(k)}} \leq 2 \dkh{1 -\underline{\sigma}^2}.
		\end{equation}
	\end{lemma}

	\begin{proof}
		We use mathematical induction to prove this lemma.
		To begin with, it follows from the inequality \eqref{eq:dist-zk} that
		\begin{equation*}
			\begin{aligned}
				\dists{Z^{(1)}}{X_i^{(0)}}
				\leq {} & \rho \sumjjd \dists{Z^{(0)}}{X_j^{(0)}}
				+ \dfrac{8}{\beta_i} \dkh{\sqrt{p} \norm{A}\fs + \mu n p} \\
				= {} & \dfrac{8}{\beta_i} \dkh{\sqrt{p} \norm{A}\fs + \mu n p}
				\leq 2 \dkh{1 - \underline{\sigma}^2},
			\end{aligned}
		\end{equation*}
		under the relationship $\beta_i > 4 (\sqrt{p} \norm{A}\fs + \mu n p ) / ( 1 - \underline{\sigma}^2 )$ 
		in Condition \ref{asp:beta}.
		Thus, the argument \eqref{eq:svd-xz} directly holds for $( Z^{(1)}, \{X_i^{(0)}\} )$.
		Now, we assume the argument holds at $( Z^{(k+1)}, \{X_i^{(k)}\} )$, 
		and investigate the situation at $( Z^{(k+2)}, \{X_i^{(k+1)}\} )$. 
		
		According to Condition \ref{asp:beta}, we have
		$12 \sqrt{p} \norm{A_i}\fs /\beta_i < 2\dkh{1 - \underline{\sigma}^2} c_i \underline{\sigma}^2$.
		Since we assume that $\dists{Z^{(k+1)}}{X_i^{(k)}} \leq 2 \dkh{1 -\underline{\sigma}^2}$,
		it follows from the relationship \eqref{eq:dist-xk+1-zk} that
		\begin{equation*}
			\begin{aligned}
				\dists{Z^{(k+1)}}{X_i^{(k+1)}}
				\leq {} & \dkh{ 1 - c_i \underline{\sigma}^2 } \dists{Z^{(k+1)}}{X_i^{(k)}}
				+ \dfrac{12}{\beta_i} \sqrt{p} \norm{A_i}\fs \\
				\leq {} & 2\dkh{1 - \underline{\sigma}^2} \dkh{1 - c_i \underline{\sigma}^2} + 2\dkh{1 - \underline{\sigma}^2} c_i\underline{\sigma}^2
				= 2\dkh{1 - \underline{\sigma}^2},
			\end{aligned}
		\end{equation*}
		which infers that $\sigma_{\min} \dkh{ (X_i^{(k+1)})\zz Z^{(k+1)} } \geq \underline{\sigma}$.
		Similar to the proof of Lemma \ref{le:des-x}, we can acquire that
		\begin{equation}\label{eq:svd-xk+1-zk}
			\norm{ \Pv{X_i^{(k+1)}} Z^{(k+1)}(Z^{(k+1)})\zz X_i^{(k+1)} }\fs 
			\geq \dfrac{1}{2} \underline{\sigma}^2 \dists{Z^{(k+1)}}{X_i^{(k+1)}}.
		\end{equation}
		Combining the condition \eqref{eq:ps-sub-x-con-2} and the equality \eqref{eq:rgrad-h}, we have
		\begin{equation}
			\label{eq:rgrad-h-xk+1}
			\begin{aligned}
				& \norm{ \Pv{X_i^{(k+1)}} H_i^{(k)} X_i^{(k+1)} }\ff 
				\leq \delta_i \norm{ \Pv{X_i^{(k)}} H_i^{(k)} X_i^{(k)} }\ff \\
				& = \delta_i \beta_i \norm{ \Pv{X_i^{(k)}} 	Z^{(k+1)}(Z^{(k+1)})\zz X_i^{(k)} }\ff
				\leq \delta_i \beta_i \distp{Z^{(k+1)}}{X_i^{(k)}}.
			\end{aligned}
		\end{equation}
		On the other hand, it follows from the triangular inequality that
		\begin{equation*}
			\begin{aligned}
				& \norm{ \Pv{X_i^{(k+1)}} H_i^{(k)} X_i^{(k+1)} }\ff \\
				\geq {} & \norm{ \Pv{X_i^{(k+1)}} \dkh{ A_i A_i\zz 
				+ \Lambda_i^{(k+1)} + \beta_i Z^{(k+1)}(Z^{(k+1)})\zz } X_i^{(k+1)} }\ff \\
				& - \norm{ \Pv{X_i^{(k+1)}} \dkh{ \Lambda_i^{(k+1)} - \Lambda_i^{(k)} } X_i^{(k+1)} }\ff 
			\end{aligned}
		\end{equation*}
		Combing the inequality \eqref{eq:svd-xk+1-zk}, it can be verified that
		\begin{equation*}
			\begin{aligned}
				& \norm{ \Pv{X_i^{(k+1)}} \dkh{A_i A_i\zz + \Lambda_i^{(k+1)} 
				+ \beta_i Z^{(k+1)}(Z^{(k+1)})\zz } X_i^{(k+1)} }\ff \\
				& = \beta_i \norm{ \Pv{X_i^{(k+1)}} Z^{(k+1)}(Z^{(k+1)})\zz X_i^{(k+1)} }\ff
				\geq \dfrac{\sqrt{2}}{2}\underline{\sigma} \beta_i \distp{Z^{(k+1)}}{X_i^{(k+1)}}.
			\end{aligned}
		\end{equation*}
		Moreover, according to Lemma B.4 in \cite{Wang2020distributed}, we have
		\begin{equation*}
			\begin{aligned}
				\norm{ \Pv{X_i^{(k+1)}} \dkh{ \Lambda_i^{(k+1)} - \Lambda_i^{(k)} } X_i^{(k+1)} }\ff
				\leq {} & \norm{ \Lambda_i^{(k+1)} - \Lambda_i^{(k)} }\ff \\
				\leq {} & 4\norm{A_i}_2^2 \distp{X_i^{(k+1)}}{X_i^{(k)}}.
			\end{aligned}
		\end{equation*}
		Combing the above three inequalities, we further obtain that
		\begin{equation*}
			\begin{aligned}
				\norm{ \Pv{X_i^{(k+1)}} H_i^{(k)} X_i^{(k+1)} }\ff
				\geq {} & \dfrac{\sqrt{2}}{2}\underline{\sigma} \beta_i \distp{Z^{(k+1)}}{X_i^{(k+1)}} \\
				& - 4\norm{A_i}_2^2 \distp{X_i^{(k+1)}}{X_i^{(k)}}.
			\end{aligned}
		\end{equation*}
		Together with \eqref{eq:rgrad-h-xk+1}, this yields that
		\begin{equation*}
			\begin{aligned}
				& \dfrac{\sqrt{2}}{2} \underline{\sigma} \beta_i \distp{Z^{(k+1)}}{X_i^{(k+1)}} \\
				& \leq \delta_i \beta_i \distp{Z^{(k+1)}}{X_i^{(k)}}
				+ 4\norm{A_i}_2^2 \distp{X_i^{(k+1)}}{X_i^{(k)}} \\
				& \leq \dkh{ \delta_i \beta_i + 4 \norm{A_i}_2^2 } \distp{Z^{(k+1)}}{X_i^{(k)}}
				+  4 \norm{A_i}_2^2 \distp{Z^{(k+1)}}{X_i^{(k+1)}}.
			\end{aligned}
		\end{equation*}
		According to Conditions \ref{asp:step} and \ref{asp:beta},
		we have $\sqrt{2} \underline{\sigma} \beta_i - 8 \norm{A_i}_2^2 > 0$
		and $\underline{\sigma} - 2\sqrt{\rho d} \delta_i > 0$.
		Thus, it can be verified that
		\begin{equation}\label{eq:dist-zk+1-xk+1}
			\begin{aligned}
				\distp{Z^{(k+1)}}{X_i^{(k+1)}}
				\leq {} & \dfrac{ 2 ( \delta_i \beta_i + 4 \norm{A_i}_2^2 ) }
				{ \sqrt{2} \underline{\sigma} \beta_i - 8 \norm{A_i}_2^2 }
				\distp{Z^{(k+1)}}{X_i^{(k)}} \\
				\leq {} & \sqrt{\dfrac{1}{2 \rho d}} \distp{Z^{(k+1)}}{X_i^{(k)}},
			\end{aligned}
		\end{equation}
		where the last inequality follows from the relationship 
		$\beta > \dfrac{4 \dkh{2 \sqrt{\rho d} + \sqrt{2}} \norm{A_i}_2^2}{\underline{\sigma} - 2 \sqrt{\rho d} \delta_i}$
		in Condition \ref{asp:beta}.
		This together with \eqref{eq:dist-zk} and \eqref{eq:svd-xz} yields that
		\begin{equation*}
			\begin{aligned}
				& \dists{Z^{(k+2)}}{X_i^{(k+1)}}
				\leq \rho \sumjjd \dists{Z^{(k+1)}}{X_j^{(k+1)}}
				+ \dfrac{8}{\beta_i} \dkh{\sqrt{p} \norm{A}\fs + \mu n p} \\
				& \leq \dfrac{1}{2d} \sumjjd \dists{Z^{(k+1)}}{X_j^{(k)}} + \dkh{1 -\underline{\sigma}^2}
				\leq \dkh{1 -\underline{\sigma}^2} + \dkh{1 -\underline{\sigma}^2}
				= 2 \dkh{1 -\underline{\sigma}^2},
			\end{aligned}
		\end{equation*}
		since we assume that $\beta_i > 8 ( \sqrt{p} \norm{A}\fs + \mu n p ) / ( 1 - \underline{\sigma}^2 )$
		in Condition \ref{asp:beta}.
		The proof is completed.
	\end{proof}

	\begin{corollary}
		\label{cor:des-xy}
		Suppose $\{Z^{(k)}\}$ is the iterate sequence generated by Algorithm \ref{alg:DSSAL1} 
		initiated from $Z^{(0)} \in \stiefel$,
		and the problem parameters satisfy Conditions \ref{asp:step} and \ref{asp:beta}. 
		Then for any $k \in \N$, we can obtain that
		\begin{equation*}
			\begin{aligned}
				& \cL( Z^{(k+1)}, \{X_i^{(k)}\}, \{\Lambda_i^{(k)}\} ) 
				- \cL( Z^{(k+1)}, \{X_i^{(k+1)}\}, \{\Lambda_i^{(k)}\} ) \\
				& \geq \dfrac{1}{4} \underline{\sigma}^2 
				\sumiid c_i\beta_i \dists{Z^{(k+1)}}{X_i^{(k)}}.
			\end{aligned}
		\end{equation*}
	\end{corollary}

	\begin{proof}
		This corollary directly follows from Lemma \ref{le:des-x} and Lemma \ref{le:svd-xz}.
	\end{proof}

	\begin{corollary}
		\label{cor:des-lambda}
		Suppose $\{Z^{(k)}\}$ is the iterate sequence generated by Algorithm \ref{alg:DSSAL1}
		initiated from $Z^{(0)} \in \stiefel$,
		and problem parameters satisfy Conditions \ref{asp:step} and \ref{asp:beta}.
		Then for any $k \in \N$, we can acquire that
		\begin{equation*}
			\begin{aligned}
				&\cL( Z^{(k+1)}, \{X_i^{(k+1)}\}, \{\Lambda_i^{(k)}\} ) 
				- \cL( Z^{(k+1)}, \{X_i^{(k+1)}\}, \{\Lambda_i^{(k+1)}\} ) \\
				& \geq - \dfrac{\sqrt{2 \rho d} + 1}{\rho d} \sumiid 
				\norm{A_i}_2^2 \dists{Z^{(k+1)}}{X_i^{(k)}}.
			\end{aligned}
		\end{equation*}
	\end{corollary}

	\begin{proof}
		According to the Cauchy–Schwarz inequality, we can show that
		\begin{equation*}
			\begin{aligned}
				& \abs{ \jkh{ \Lambda_i^{(k+1)} - \Lambda_i^{(k)}, \dsp{X_i^{(k+1)}}{Z^{(k+1)}} } }
				\leq \norm{  \Lambda_i^{(k+1)} - \Lambda_i^{(k)} }\ff
				\distp{Z^{(k+1)}}{X_i^{(k+1)}} \\
				& \leq \sqrt{\dfrac{8}{\rho d}} \norm{A_i}_2^2 \distp{X_i^{(k+1)}}{X_i^{(k)}} 
				\distp{Z^{(k+1)}}{X_i^{(k)}},
			\end{aligned}
		\end{equation*}
		where the last inequality follows from Lemma B.4 in \cite{Wang2020distributed} 
		and \eqref{eq:dist-zk+1-xk+1}.
		In addition, we have
		\begin{equation*}
			\begin{aligned}
				\distp{X_i^{(k+1)}}{X_i^{(k)}} 
				\leq {} & \distp{Z^{(k+1)}}{X_i^{(k+1)}}  + \distp{Z^{(k+1)}}{X_i^{(k)}} \\
				\leq {} & \dfrac{\sqrt{2 \rho d} + 1}{\sqrt{2 \rho d}} \distp{Z^{(k+1)}}{X_i^{(k)}},
			\end{aligned}
		\end{equation*}
		which implies that
		\begin{equation*}
			\begin{aligned}
				& \jkh{ \Lambda_i^{(k+1)} - \Lambda_i^{(k)}, \dsp{X_i^{(k+1)}}{Z^{(k+1)}} } \\
				& \geq - \dfrac{2 \dkh{ \sqrt{2 \rho d} + 1 }}{\rho d} \norm{A_i}_2^2 
				\dists{Z^{(k+1)}}{X_i^{(k)}}.
			\end{aligned}
		\end{equation*}
		Combing the fact that
		\begin{equation*}
			\begin{aligned}
				& \cL(  Z^{(k+1)}, \{X_i^{(k+1)}\}, \{\Lambda_i^{(k)}\} ) 
				- \cL(  Z^{(k+1)}, \{X_i^{(k+1)}\}, \{\Lambda_i^{(k+1)}\} ) \\
				& = \dfrac{1}{2} \sumiid \jkh{ \Lambda_i^{(k+1)}-\Lambda_i^{(k)}, 
				\dsp{X_i^{(k+1)}}{Z^{(k+1)}} },
			\end{aligned}
		\end{equation*}
		we complete the proof.
	\end{proof}
	
	Now based on these lemmas and corollaries, we can demonstrate 
	the monotonic non-increasing of $\hkh{ \cL( \{X_i^k\}, Z^k, \{\Lambda_i^k\} ) }$,
	which results in the global convergence of our algorithm.

	\begin{proposition}
		\label{prop:lag-des}
		Suppose $\{Z^{(k)}\}$ is the iteration sequence generated by Algorithm \ref{alg:DSSAL1} 
		initiated from $Z^{(0)} \in \stiefel$,
		and problem parameters satisfy Conditions \ref{asp:step} and \ref{asp:beta}.
		Then the sequence of augmented Lagrangian functions 
		$\{ \cL( Z^{(k)}, \{X_i^{(k)}\}, \{\Lambda_i^{(k)}\} ) \}$ is monotonically non-increasing, 
		and for any $k \in \N$, it satisfies the following sufficient descent property:
		\begin{equation}
			\label{eq:lag-des}
			\begin{aligned}
				& \cL( Z^{(k)}, \{X_i^{(k)}\}, \{\Lambda_i^{(k)}\} ) 
				- \cL( Z^{(k+1)},  \{X_i^{(k+1)}\}, \{\Lambda_i^{(k+1)}\} ) \\
				& \geq \sumiid J_i \dists{Z^{(k+1)}}{X_i^{(k+1)}} 
				+ \bar{M} \norm{D^{(k)}}\fs,
			\end{aligned}
		\end{equation}
		where $J_i = \dfrac{1}{2}  \rho d \underline{\sigma}^2 
		c_i \beta_i - 2 (\sqrt{2 \rho d} + 1) \norm{A_i}_2^2 > 0$
		is a constant.
	\end{proposition}

	\begin{proof}
		Combining Corollary \ref{cor:des-z}, Corollary \ref{cor:des-xy}, 
		and Corollary \ref{cor:des-lambda}, we obtain that
		\begin{equation*}
			\begin{aligned}
				& \cL( Z^{(k)}, \{X_i^{(k)}\}, \{\Lambda_i^{(k)}\} ) 
				- \cL( Z^{(k+1)}, \{X_i^{(k+1)}\}, \{\Lambda_i^{(k+1)}\} ) \\
				& \geq \sumiid \dkh{ \dfrac{1}{4} \underline{\sigma}^2 c_i \beta_i 
				- \dfrac{\sqrt{2 \rho d} + 1}{\rho d} \norm{A_i }_2^2 } 
				\dists{Z^{(k+1)}}{X_i^{(k)}}+ \bar{M} \norm{D^{(k)}}\fs.
			\end{aligned}
		\end{equation*}
		Recalling the relationship $\beta_i > 4 ( \sqrt{2 \rho d} + 1) \norm{A_i}_2^2 / 
		(\rho d \underline{\sigma}^2 c_i)$ 
		in Condition \ref{asp:beta},
		we can conclude that $\cL( Z^{(k)}, \{X_i^{(k)}\}, \{\Lambda_i^{(k)}\} ) \geq \cL( Z^{(k+1)}, \{X_i^{(k+1)}\}, \{\Lambda_i^{(k+1)}\} )$.
		Hence, the sequence $\{\cL( Z^{(k)}, \{X_i^{(k)}\}, \{\Lambda_i^{(k)}\} )\}$ is monotonically non-increasing.
		Finally, the above relationship together with \eqref{eq:dist-zk+1-xk+1} 
		yields the assertion \eqref{eq:lag-des}.
		The proof is finished.
	\end{proof}

	Based on the above properties, we are ready to prove Theorem \ref{thm:global}, 
	which establishes the global convergence rate of our proposed algorithm. 

	\begin{proof}[Proof of Theorem \ref{thm:global}]
		The whole sequence $\{ Z^{(k)}, \{X_i^{(k)}\} \}$ is naturally bounded,
		since each of $X_i^{(k)}$ or $Z^{(k)}$ is orthogonal.
		Then it follows from the Bolzano-Weierstrass theorem that 
		this sequence exists an accumulation point $\{Z^{\ast}, \{X_i^{\ast}\}\}$, 
		where $Z^{\ast} \in \stiefel$ and $X_i^{\ast} \in \stiefel$.
		Moreover, the boundedness of $\{\Lambda_i^{(k)}\}$ 
		results from the multipliers updating formula  \eqref{eq:ps-mult}.
		Hence, the lower boundedness of $\{ \cL( Z^{(k)}, \{ X_i^{(k)}\}, \{\Lambda_i^{(k)}\} ) \}$
		is owing to the continuity of the augmented Lagrangian function.
		Namely, there exists a constant $\underline{L}$ such that
		$\cL ( Z^{(k)}, \{X_i^{(k)}\}, \{\Lambda_i^{(k)}\} ) \geq \underline{L}$ for all $k \in \N$.
		
		It follows from the sufficient descent property \eqref{eq:lag-des} that
		\begin{equation}
			\label{eq:sublinear-opt}
			\begin{aligned}
				& \sum\limits_{k=1}^{K} \norm{D^{(k)}}\fs \\
				& \leq \bar{M}\inv \sum\limits_{k=1}^{K} 
				\dkh{ \cL ( Z^{(k)}, \{X_i^{(k)}\}, \{\Lambda_i^{(k)}\} ) 
					- \cL ( Z^{(k+1)}, \{X_i^{(k+1)}\}, \{\Lambda_i^{(k+1)}\} ) } \\
				& = \bar{M}\inv \dkh{ \cL ( Z^{(1)}, \{X_i^{(1)}\}, \{\Lambda_i^{(1)}\} ) 
					- \cL ( Z^{(K+1)}, \{X_i^{(K+1)}\}, \{\Lambda_i^{(K+1)}\} ) } \\
				& \leq \bar{M}\inv 
				\dkh{ \cL ( Z^{(1)}, \{X_i^{(1)}\}, \{\Lambda_i^{(1)}\} ) - \underline{L} },
			\end{aligned}
		\end{equation}
		and
		\begin{equation}
			\label{eq:sublinear-fea}
			\begin{aligned}
				& \sum\limits_{k = 1}^{K} \dists{Z^{(k)}}{X_i^{(k)}} \\
				& \leq J_i\inv \sum\limits_{k=1}^{K} 
				\dkh{ \cL ( Z^{(k-1)}, \{X_i^{(k-1)}\}, \{\Lambda_i^{(k-1)}\} ) 
					- \cL ( Z^{(k)},  \{X_i^{(k)}\}, \{\Lambda_i^{(k)}\} ) } \\
				& = J_i\inv \dkh{ \cL ( Z^{(0)}, \{X_i^{(0)}\}, \{\Lambda_i^{(0)}\} ) 
					- \cL ( Z^{(K)}, \{X_i^{(K)}\}, \{\Lambda_i^{(K)}\} ) } \\
				& \leq J_i\inv \dkh{ \cL ( Z^{(0)}, \{X_i^{(0)}\}, \{\Lambda_i^{(0)}\} ) - \underline{L} }.
			\end{aligned}
		\end{equation}
		Upon taking the limit as $K \to \infty$, we obtain that
		\begin{equation*}
			\sum\limits_{k = 1}^{\infty} \norm{D^{(k)}}\fs < \infty 
			\text{~~and~~} 
			\sum\limits_{k = 1}^{\infty} \dists{Z^{(k)}}{X_i^{(k)}} < \infty,			
		\end{equation*}
		which further implies that
		\begin{equation*}
			\lim\limits_{k\to\infty} \norm{ D^{(k)} }\ff = 0
			\text{~~and~~}
			\lim\limits_{k\to\infty} \distp{Z^{(k)}}{X_i^{(k)}} = 0,
		\end{equation*}
		respectively.
		Combing this with Lemma \ref{le:optimality}, 
		we know that any accumulation point $Z^{\ast}$ of sequence $\{Z^{(k)}\}$ 
		is a first-order stationary point of the problem \eqref{eq:opt-spca-l1}.
		
		Eventually, we prove the sublinear convergence rate.
		Indeed, it follows from the inequalities \eqref{eq:sublinear-opt} 
		and \eqref{eq:sublinear-fea} that
		\begin{equation*}
			\begin{aligned}
				& \min\limits_{k = 1, \dotsc, K} \hkh{ \norm{D^{(k)}}\fs 
				+\dfrac{1}{d} \sumiid \dists{Z^{(k)}}{X_i^{(k)}} } \\
				& \leq \dfrac{1}{K} \sum\limits_{k = 1}^K 
				\hkh{ \norm{D^{(k)}}\fs +\dfrac{1}{d} \sumiid \dists{Z^{(k)}}{X_i^{(k)}} }
				\leq \dfrac{C}{K},
			\end{aligned}
		\end{equation*}
		where 
		\begin{equation*}
			\begin{aligned}
				C = {} & \bar{M}\inv 
				\dkh{ \cL ( Z^{(1)}, \{X_i^{(1)}\}, \{\Lambda_i^{(1)}\} ) - \underline{L} } \\
				& + \dkh{ \sumiid J_i\inv}
				d\inv \dkh{ \cL ( Z^{(0)}, \{X_i^{(0)}\}, \{\Lambda_i^{(0)}\} ) - \underline{L} }
			\end{aligned}
		\end{equation*}
		is a positive constant.
		This completes the proof.
	\end{proof}

\end{appendices}

\bibliographystyle{siam}
\bibliography{library}
\addcontentsline{toc}{section}{References}

\end{document}